\documentclass[a4paper,11pt]{amsart}
\usepackage{mathrsfs}
\usepackage[all]{xy}
\usepackage{amsmath,amssymb,amscd,bbm,amsthm,mathrsfs}
\usepackage{graphicx}
\newtheorem{thm}{Theorem}[section]
\newtheorem{lem}{Lemma}[section]
\newtheorem{cor}{Corollary}[section]
\newtheorem{prop}{Proposition}[section]
\newtheorem{rem}{Remark}[section]

\begin{document}
\numberwithin{equation}{section}

\title[The  Yamabe operator and invariants on OC manifolds]{The  Yamabe operator and invariants on octonionic contact manifolds and convex cocompact subgroups of ${\rm F}_{4(-20)}$}
\author[Yun Shi and  Wei Wang]{Yun Shi$^\dagger$ and  Wei Wang$^\ast$ }

\begin{abstract}
An octonionic contact  (OC) manifold is always  spherical.  We construct the OC Yamabe operator on an OC manifold and prove its transformation formula under  conformal OC transformations. An OC manifold is scalar positive, negative or vanishing if and only if its   OC Yamabe invariant is positive, negative or zero, respectively. On a scalar positive OC manifold, we can construct the Green function of the  OC  Yamabe operator, and apply it to construct a conformally invariant tensor. It becomes an OC metric if the OC positive mass conjecture is true. We also show the connected sum of two scalar positive OC manifolds to be scalar positive if the neck is sufficiently long. On the  OC manifold   constructed from a  convex cocompact subgroup  of ${\rm F}_{4(-20)}$, we construct a  Nayatani type Carnot-Carath\'eodory metric. As a corollary,  such an  OC manifold is scalar positive, negative or  vanishing if and only if the Poincar\'e  critical exponent of the  subgroup is less than,  greater than or equal to  $10,$  respectively.
\end{abstract}
\keywords{  octonionic contact (OC)  manifolds;  the Biquard connection; the  OC  Yamabe operator; transformation formula under conformal OC transformations; convex cocompact subgroups of ${\rm F}_{4(-20)}$;   the octonionic Heisenberg group}
\thanks{The first author is partially supported by National Nature Science
Foundation
  in China (Nos. 11801508,  11971425); The second author is partially supported by National Nature Science
Foundation
  in China (No. 11971425)\\
$^\dagger$Department of Mathematics, Zhejiang University of Science and Technology, Hangzhou 310023, China,
E-mail: shiyun@zust.edu.cn;
 $^\ast$Department of Mathematics, Zhejiang University,  Hangzhou 310027, China,
E-mail: wwang@zju.edu.cn}
\maketitle

\section{Introduction}

Biquard \cite{biquard} introduced notions of quaternionic and octonionic contact manifolds. Recently, it is an active direction to study  quaternionic contact manifolds (see \cite{Barilari} \cite{Ivanov7}-\cite{Ivanov8} \cite{ Minchev} \cite{Shi} \cite{wang2} and reference therein).   An  \emph{octonionic contact manifold} $(M,g,\mathbb{I})$ is a $15$-dimensional manifold  $M$ with a codimension $7$ distribution $H$
locally given as the kernel of a $\mathbb{R}^{7}$-valued $1$-form $\Theta=(\theta_{1},\cdots,\theta_{7})$, on which $g$ is a Carnot-Carath\'eodory metric, where  $\mathbb{I}:=({I_{1},\cdots,I_{7}})$ with $I_\beta\in$End$(H)$ satisfying  the octonionic commutating relation (\ref{Nij}). Note that $\mathbb I$ is a rank-$7$ bundle
\begin{align*}
\mathbb{I}=\{a_1I_{1}+\cdots+a_7I_{7}|a_1^{2}+\cdots+a_7^{2}=1\},
\end{align*}
which consists of endomorphisms of $H$ locally generated by $7$ almost complex structures $I_{1},\cdots,I_{7}$ on $H.$ They are hermitian compatible with the metric:
\begin{align}\label{2.01}
g(I_{\beta}\cdot,I_{\beta}\cdot)=g(\cdot,\cdot),
\end{align} and satisfy  the  compatibility condition
\begin{align}\label{2,1}
g(I_{\beta}X,Y)=\hbox{d}\theta_{\beta}(X,Y),
\end{align}   for any  $X,Y\in H, \beta=1,\cdots,7$. Biquard  \cite{biquard} introduced a canonical connection on an OC manifold, which is  called Biquard connection now.

As pointed by Biquard  \cite{biquard}, any octionionic contact manifold is spherical by a theorem of Yamaguchi \cite{Yamaguchi}.  $f:(M,g,\mathbb I)\rightarrow(M',g',\mathbb I')$ is called {\it$OC$ conformal} if locally we have $f^*g'=\phi g$ for some positive function $\phi>0$ and $f^*\mathbb I'=\Psi\mathbb I$ for some ${\rm SO}(7)$ valued smooth function $\Psi$ \cite{biquard}.
The conformal  class of OC manifolds is denoted by $\left[M,g,\mathbb{I}\right].$ The purpose of this paper is to investigate the conformal geometry of OC manifolds as we have done for spherical CR manifolds \cite{wang1} and spherical qc manifolds \cite{Shi}.
The main difficulty to investigate OC manifolds comes from algebra. The associative algebras $\mathbb R,\mathbb C$ and $\mathbb H$ are replaced by the non-associative octonion  algebra $\mathbb O,$ and the classical Lie groups ${\rm SO,\ SU}$ and ${\rm Sp}$ are replaced by the exceptional Lie group ${\rm F}_{4(-20)}.$ For example, it is more complicated  to describe explicitly  actions of ${\rm F}_{4(-20)}$ on the octionionic hyperbolic space as isometries   and on the octonionic Heisenberg group as conformal transformations, which are given in Section 2.

We give the transformation formula of scalar curvatures under a conformal transformation in Section 3.
\begin{thm}\label{s-curva}
The scalar curvature $s_{\tilde{g}}$ of the Biquard connection of  $(M,\tilde{g},\mathbb{I})$ with $\tilde{g}=\phi^{\frac{4}{Q-2}}g$ satisfies the OC Yamabe equation:
\begin{align}\label{122}
L_g\phi=s_{\tilde{g}}\phi^{\frac{Q+2}{Q-2}},\quad b=\frac{4(Q-1)}{Q-2}=\frac{21}{5},
\end{align}where $L_{g}:=b\Delta_{g}+s_{g}$ is the OC Yamabe operator, and $\Delta_g$ is the SubLaplacian
associated to the Carnot-Carath\'eodory metric $g$  and  $Q=22$ is the homogeneous dimension of $M.$
\end{thm}

\begin{cor}\label{L-trans}
The   OC Yamabe operator $L_{g}$ satisfies the transformation formula
\begin{align}\label{o121}
L_{\tilde{g}}f=\phi^{-\frac{Q+2}{Q-2}}L_{g}(\phi f),
\end{align}
if $\tilde{g}=\phi^{\frac{4}{Q-2}}g$ and $f\in C^{\infty}(M)$.
\end{cor}
As in the locally  conformally flat, CR and qc cases, for a connected compact octonionic  manifold $(M,g,\mathbb{I})$, we have the following  trichotomy: there exists   an OC metric $\tilde{g}$ conformal to $g$  which has  either positive,   negative or  vanishing scalar curvature  everywhere. Denote by $G_{g}(\xi,\cdot)$ the Green function of the  OC Yamabe operator with the pole at $\xi$, i.e. $ L_{g}G_{g}(\xi,\cdot)=\delta_{\xi}, $
where $\delta_{\xi}$ is the Dirac function at the point $\xi.$
On a scalar positive OC manifold, the Green function of the OC Yamabe operator $L_g$ always exists, and
\begin{align}\label{222} \rho_{g}(\xi,\eta)&=\frac{1}{\phi(\xi)\phi(\eta)}
\cdot\frac{C_{Q}}{\|\xi^{-1}\eta\|^{Q-2}},\quad \xi,\eta\in \mathscr{H},
\end{align}
is its singular part, if we identify a neighborhood of $\xi$ with an open set of the octonionic Heisenberg group $\mathscr H$ with the OC metric
$g=\phi^{\frac{4}{Q-2}}g_{0}.$ Here $g_{0}$ is the standard OC metric on the octonionic  Heisenberg group $\mathscr{H}$, $\|\cdot\|$ is the norm on
the octonionic Heisenberg group and $C_{Q}$ is a positive constant  (\ref{146}).
\begin{thm}\label{can-g}
Let $(M,g,\mathbb{I})$ be connected, compact, scalar positive OC manifold, which is not OC equivalent to the standard sphere. Define
$
can(g):=\mathcal{A}_{g}^{2}g,
$
where
\begin{align}
\mathcal{A}_{g}(\xi)=\lim_{\eta\rightarrow\xi}\left|G_{g}(\xi,\eta)-
\rho_{g}(\xi,\eta)\right|^{\frac{1}{Q-2}},
\end{align}
if $g=\phi^{\frac{4}{Q-2}}g_{0}$ on a neighborhood $U$ of $\xi$.  Then, $can(g)$ is well-defined and depends only on the conformal class $\left[M,g,\mathbb I\right].$
\end{thm}

As in the locally flat \cite{yau},  CR \cite{paul,Li} and qc \cite{Shi} cases, we propose the following {\it {OC positive mass conjecture}}: {\it
Let $(M,g,\mathbb{I})$ be a compact scalar positive  OC manifold with ${\rm dim}\ M=15$. Then,\\
\noindent 1. For each $\xi\in M$, there exists a local OC diffeomorphism $C_{\xi}$ from a neighborhood of $\xi$ to the octonionic Heisenberg group
$\mathscr{H}$ such that $C_{\xi}(\xi)=\infty$ and  $$\left(C_{\xi}^{-1}\right)^{*}\left(G_{g}(\xi,\cdot)
^{\frac{4}{Q-2}}g\right)
=h^{\frac{4}{Q-2}}g_{0},$$ where  $$h(\eta)=1+A_{g}(\xi)\|\eta\|^{-Q+2}+O(\|\eta\|^{-Q+1}),$$
near $\infty,$ and $g_{0}$ is the standard OC metric on $\mathscr{H}.$
$A_{g}(\xi)$ is called the {OC mass} at the point $\xi$.

\noindent 2. $A_{g}(\xi)$ is nonnegative. It is zero if and only if $(M,g,\mathbb{I})$ is OC equivalent to the standard sphere.}
\vskip 3mm

We  also introduce the connected sum of two OC manifolds and prove that the connected sum of two scalar positive OC manifolds is also scalar positive if the neck is sufficiently long.

In the last section, we recall  definitions of  a convex cocompact subgroup of ${\rm F}_{4(-20)}$ and the Patterson-Sullivan measure. For a discrete  subgroup $\Gamma$ of ${\rm F}_{4(-20)},$ the \emph{limit set} of $\Gamma$ is
\begin{align}\label{limitset}
\Lambda(\Gamma)=\overline{\Gamma{q}}\cap \partial \mathcal U,
\end{align}
for $q$ in the Siegel domain $\mathcal U,$ where $\overline{\Gamma{q}}$ is the closure of the orbit of $q$ under $\Gamma,$ and
\begin{align}\label{Omegag}
\Omega(\Gamma)= \partial\mathcal U\backslash\Lambda(\Gamma)
\end{align}
is the maximal open set where $\Gamma$ acts discontinuously.  It is known that $\Omega(\Gamma)/\Gamma$ is a compact OC manifold  when $\Gamma$
is a convex cocompact subgroup of ${\rm F}_{4(-20)}$.
The \emph{Poincar\'e critical exponent} $\delta(\Gamma)$ of a discrete subgroup $\Gamma$ is defined as
\begin{align*}
\delta(\Gamma)=\inf\left\{s>0;\sum_{\gamma\in\Gamma}e^{-\frac{s}{2}\cdot d(z,\gamma (w))}<\infty\right\},
\end{align*}
where $z$ and $w$ are two points in $\mathcal U$ and $d(\cdot,\cdot)$ is the octonionic hyperbolic distance on $\mathcal U$. $\delta(\Gamma)$
is independent of the particular choice of points $z$ and $w$.
For any convex cocompact subgroup $\Gamma$ of ${\rm F}_{4(-20)}$, there exists a probability measure $\tilde{\mu}_{\Gamma}$ supported on 	its limit set $\Lambda(\Gamma),$ called  the  \emph{Patterson-Sullivan
measure}, such that
\begin{align*}
\gamma^{*}\tilde{\mu}_{\Gamma}=|\gamma'|^{\delta(\Gamma)}\tilde{\mu}_{\Gamma}
\end{align*}
for any $\gamma\in\Gamma$ (cf. \cite{Corlette}), where $|\gamma'|$ is the conformal factor. Define $\mu_\Gamma=\chi\tilde \mu_\Gamma$ by choosing a suitable factor $\chi$ in (\ref{p1}). The conformal factor of $\mu_\Gamma$ is the same as that $\Gamma$ acts on the octonionic Heisenberg group.
We define a $C^{\infty}$ function on $\Omega(\Gamma)$ by
\begin{align}\label{phigamma}
\phi_{\Gamma}(\xi):=\left(\int_{\Lambda(\Gamma)}
G_{0}^{\kappa}(\xi,\zeta)\hbox{d} \mu_{\Gamma}(\zeta)\right)^{\frac{1}{\kappa}},\
\kappa=\frac{2\delta(\Gamma)}{Q-2},
\end{align}where $G_0(\xi,\zeta)$ is the Green function of OC Yamabe operator on the octonionic  Heisenberg group with the pole at $\xi.$ Then \begin{align}\label{ggamma}
g_{\Gamma}:=\phi_{\Gamma}^{\frac{4}{Q-2}}  g_{0},
\end{align}
is  invariant under $\Gamma,$  which is the OC generalization of Nayatani's canonical metric in conformal geometry \cite{Nayatani}. See \cite{Naya,wang1} and \cite{Shi} for CR case and qc case, respectively.
\begin{thm}\label{6.1}
Let $\Gamma$ be a convex cocompact subgroup of ${\rm F}_{4(-20)}$ such that $\Lambda(\Gamma)\neq\{point\}$. Then, the scalar curvature of $(\Omega(\Gamma)/\Gamma,{g}_{\Gamma},\mathbb{I})$ is positive {\rm(}or
negative, or zero{\rm)} everywhere if and only if  $\delta(\Gamma)<10$ {\rm(}or
$\delta(\Gamma)>10$, or $\delta(\Gamma)=10${\rm)}.
\end{thm}

\section{Preliminaries of octonions, Jordan algebra, ${\rm F}_{4(-20)},$ the octionionic hyperbolic space and the octonionic Heisenberg group}
\subsection{Octonions and Jordan algebra}
In this paper we denote by   $a,b,c,\cdots$ numbers in $\mathbb{Z}_{8}:=\{0,\cdots,7\}$ and by  $\alpha,\beta,\gamma,\cdots$ numbers in $\mathbb{Z}_{8+}:=\{1,\cdots,7\}.$
The octonion algebra has a basis $\{e_a\}$  satisfying the relation $e_ae_{0}=e_{0}e_a, a \in \mathbb{Z}_{8}$, and
\begin{align}\label{2.1'}
e_{\alpha}e_{\beta}=-\delta_{\alpha\beta}+\epsilon_{\alpha\beta\gamma}e_{\gamma},
\end{align}
$\alpha,\beta\in\mathbb{Z}_{8+}.$ The constants $\epsilon_{\alpha\beta\gamma}$ in (\ref{2.1'}) are completely antisymmetric in $\alpha,\beta,\gamma,$ and equal the value $+1$ for $(\alpha,\beta,\gamma)\in\Omega$ (cf. e.g. \cite{Wang H}), where
\begin{align}\label{2.2}
\Omega:=\{(1, 2, 3),(2, 4, 6),(4, 3, 5),(3, 6, 7),(6, 5, 1),(5, 7, 2),(7, 1, 4)\}.
\end{align}
$\epsilon_{\alpha\beta\gamma}$ is nonzero only for  $(\alpha,\beta,\gamma)$ to be a
permutation of a triple  in $\Omega.$ We may choose   $\Omega$ differently (cf. e.g. \cite{Baez}).   By (\ref{2.2}), it is direct to see that, for fixed $(\alpha,\beta),$ $\epsilon_{\alpha\beta\gamma}$ is nonvanishing for only one $\gamma.$ So  $\epsilon_{\alpha\beta\gamma}e_{\gamma}$ in (\ref{2.1'}) is the same as  $\sum_{\gamma=1}^{7}\epsilon_{\alpha\beta\gamma}e_{\gamma}.$

The octonion  algebra $\mathbb{O}$ is neither commutative nor associative. For $x,y,z\in\mathbb{O},$ define an {\it associator}: $\{x,y,z\}:=(xy)z-x(yz).$ Besides, the octonions obey some weak associative laws, such as the so-called
{\it Moufang identities} (cf. e.g. (2.5) in \cite{Wang H}):
\begin{align*}
(uvu)x = u(v(ux)),\quad x(uvu) = ((xu)v)u,\quad u(xy)u = (ux)(yu),
\end{align*}
for any $u,v,x,y\in\mathbb{O}.$ In particular, we have
\begin{align*}
uvu = (uv)u = u(vu).
\end{align*}
Namely,  the octonion  algebra
is {\it alternative}, i.e. $\{u,v,u\}=\{u,u,v\}=\{v,u,u\}=0$ for any $u,v\in\mathbb{O}$.
\begin{prop}\label{abc}{\rm (cf. }\cite[Proposition 2.1]{Wang H} {\rm and references therein)}
For any $a,b,c,d\in\mathbb{Z}_8,$ we have
\begin{align}\label{mnps}
\{e_a,e_b,e_c\}:=(e_ae_b)e_c-e_a(e_be_c)=2\epsilon_{abcd}e_d,
\end{align}
where $\epsilon_{abcd}$ is totally antisymmetric, and equals to $1$ for $(a, b, c, d)\in\Lambda,$ where
\begin{align*}
\Lambda= \{(5, 4, 6, 7),(7, 3, 5, 1),(1, 6, 7, 2),(2, 5, 1, 4),(4, 7, 2, 3), (3, 1, 4, 6), (6, 2, 5, 3)\}.
\end{align*}
$\epsilon_{abcd}$ is nonzero only when $(a, b, c, d)$ is the permutation of a quadruple in $\Lambda.$
\end{prop}
Denote by $M(n,\mathbb{O})$ all $(n\times n)$ matrices with entries in $\mathbb{O}$ and  $I_{n}={\rm diag}(1,\cdots,1)\in M(n,\mathbb{O})$  the unit matrix.
For $A\in M(n,\mathbb{O})$, $A^{t}$ denotes the transposed matrix of $A.$ The {\it Jordan algebra}
\begin{align*}
\mathcal{J}:=\{X\in M(3,\mathbb{O})|D_1 X=\bar{X}^{t}D_1\},
\end{align*}
has ${\rm F}_{4(-20)}$ as its automorphism group (cf. \cite{Allcock,Mostow,Parker}), where \begin{equation}\label{d1}
D_{1}=\begin{pmatrix}0 &0 &1\\0 &1&0\\1&0&0 \end{pmatrix}.
\end{equation}
\begin{prop}{\rm(\cite[Lemma 14.66]{Harvey})}
{\rm Spin(7)} is generated by   $\{L_\mu:\mu\in S^6\subset{\rm Im}\ \mathbb{O}\},$ where $L_\mu x=\mu x,$ for $x\in\mathbb O.$
\end{prop}
It is characterized as the subgroup of ${\rm SO}(8)$ which conjugates $\mathbb{R}^7=\{L_\mu;\mu\in{\rm Im}\ \mathbb{O}\}$ to itself (cf. \cite{biquard}). We have a decomposition of $\mathfrak{so}(8)$ into irreducible $\mathfrak{so}(7)$ modules in the form \begin{align}\label{so7}
\wedge^2\mathbb{O}=\mathfrak{so}(8)=\mathfrak{so}(7)\oplus\mathbb{R}^7.\end{align}

In the sequel we will use the  Einstein convention of repeated indices. If we identify $\mathbb{O}$ with $\mathbb{R}^8$,  the left multiplication by $e_\beta$ is a linear transformation on $\mathbb{R}^8,$ given by a $(8\times8)$-matrix $I_\beta=L_{e_\beta}.$ Namely, by identifying ${x}=  x_ae_a$ with ${x}=(x_0,\cdots,x_7)^t,$ we have
\begin{align*}
e_\beta {x}=(I_\beta x)_ae_a.
\end{align*}
$(I_\beta x)_a$ is the $a$-th entry of the vector  in $\mathbb{R}^8$ corresponding to $e_\beta x.$ $I_1,\cdots,I_7$ do not satisfy the commutating  relation (\ref{2.1'}) of octonions  because of the non-associativity of $\mathbb O.$
\begin{prop}{\rm(}cf. \cite[Proposition 3.1, 3.2]{Wang H}{\rm)}
Suppose $e_a e_b=e_\beta $, $a,b\in\mathbb{Z}_{8+} .$ Then we have
\begin{align}\label{Nij}
I_a I_b=I_\beta-N^{ab},
\end{align}
where $N^{ab}$  are $(8\times 8)$-matrices with $\left(N^{ab}\right)_{dc}=2\epsilon_{abcd}.$
\end{prop}
\begin{proof}
Note that $(e_a e_b)x=e_\beta x=(I_\beta x)_ce_c,$ and
$$e_a (e_b x)=e_a\left((I_b x)_ce_c\right)=\left(I_bx\right)_c\cdot\left(I_a\right)_{dc}e_d=(I_aI_b x)_de_d.$$
We find that
\begin{equation*}\begin{aligned}
(I_aI_b x)_de_d=&e_a(e_b e_c)x_c=(e_a e_b) e_cx_c-2\epsilon_{abcd}e_d x_c\\=&e_\beta \left(x_c e_c\right)-2\epsilon_{abcd}e_dx_c=\left(I_\beta x\right)_de_d -\left(N^{ab}x\right)_d e_d.
\end{aligned}\end{equation*}
Then (\ref{Nij}) follows.
\end{proof}
Set \begin{equation}\label{o2.12}
E^{\beta}=\begin{pmatrix} 0 & -\nu_{\beta} \\ \nu^{t}_{\beta} & \varepsilon^{\beta} \end{pmatrix},\quad \beta=1,\cdots,7.
\end{equation}
Here $\varepsilon^{\beta}$ are $(7\times7)$-matrices with
\begin{align*}
\varepsilon^{\beta}_{\alpha\gamma}=\epsilon_{\gamma\alpha\beta},\quad \nu_{\beta}=(\delta_{\beta1},\cdots,\delta_{\beta7})\in\mathbb{R}^{7}.
\end{align*}
Note that $e_\beta x_\gamma e_\gamma=\epsilon_{\beta\gamma\alpha}x_\gamma e_\alpha-x_\beta=\epsilon_{\gamma\alpha\beta}x_\gamma e_\alpha-x_\beta=\left(\varepsilon_{\alpha\gamma}^\beta x_\gamma\right) e_\alpha-x_\beta.$ So $(I_\beta)_{ba}$ is the $(b,a)$-th entry of $E^\beta$ given by  (\ref{o2.12}).
 \begin{prop}\label{p2.4} $E^{\beta}$'s are antisymmetric matrices    satisfying
{\rm(1)} $\left(E^\beta\right)^2=-I_{8\times8};$
{\rm(2)} $E^\alpha E^\beta=-E^\beta E^\alpha,$ for $\alpha\neq\beta\in\mathbb Z_{8+};$
 \end{prop}
\begin{proof}
The proposition  is  proved in \cite[Proposition 3.1, 3.3]{Wang H}.\\
(1)  $-x=(e_\beta e_\beta)x=e_\beta (e_\beta x)=e_\beta (E^\beta x)_ae_a=((E^\beta)^2x)_ae_a,$
i.e. $\left(E^\beta\right)^2=-I_{8\times8}.$\\
(2) It follows form (\ref{Nij}) that $E^\alpha E^\beta$ are also antisymmetric by the antisymmetry of $E^\beta$ and $N.$
\end{proof}

\subsection{${\rm F}_{4(-20)}$ and the octonionic hyperbolic space}
Let us recall basic facts  of octonionic hyperbolic space (cf. \cite{Allcock,Parker,Platis}).
Define \begin{equation*}
\mathbb{O}^3_0=\left\{\textbf{v}=\begin{pmatrix}y\\x\\z \end{pmatrix}:x,y,z\ {\rm all\ lie\ in\ some\ associative\ subalgebra\ of}\ \mathbb{O}\right\}.
\end{equation*}
$\textbf{v}\sim \textbf{w}$ if $\textbf{v}=\textbf{w}\lambda$ for some $\lambda$ in an associative subalgebra of $\mathbb{O}$ containing the entries $x,y,z$ of $\textbf{v}.$ The map from $\mathbb{O}^3_0$ to the set of equivalent classes is the analogue of right projection and so we denote by $P\mathbb{O}^3_0$ the set of right equivalent classes. Define a map $\pi_{D_1}:\mathbb{O}^3_0\rightarrow\mathcal J$ by
\begin{equation*}\begin{aligned}\pi_{D_1}(\mathbf v)=\mathbf v\mathbf v^*D_{1}=\begin{pmatrix}y\bar z &y\bar{x} &|y|^{2}\\x\bar z &|x|^{2} &x\bar{y}\\|z|^2&z\bar{x}&z\bar{y} \end{pmatrix},
\end{aligned}\end{equation*}
where $D_{1}$ is given by $(\ref{d1})$ and $\mathbf v^*=(\bar y,\bar x,\bar z).$ If $v=(x,y)\in\mathbb{O}^{2},$  let $\tilde{v}$ denote the column vector
\begin{equation*}
\tilde{v}=\begin{pmatrix}y\\x\\1 \end{pmatrix}\in\mathbb O^3.
\end{equation*}

Set $j:\mathbb O^2\rightarrow M(3,\mathbb O)$ with \begin{equation}\begin{aligned}\label{o2.49}j(v)= \pi_{D_1}(\tilde v)=\tilde{v}\tilde{v}^*D_{1}=\begin{pmatrix}y &y\bar{x} &|y|^{2}\\x &|x|^{2} &x\bar{y}\\1&\bar{x}&\bar{y} \end{pmatrix}, \qquad j(\infty)= \pi_{D_1}\begin{pmatrix}1\\0\\0 \end{pmatrix}=\begin{pmatrix}0 &0&1\\0 &0 &0\\0&0&0 \end{pmatrix},
\end{aligned}\end{equation}
for $v=(x,y)\in\mathbb O^2.$
Then we can define
\begin{equation*}
\begin{aligned}
D_{\pm}:=\left\{\tilde{v}\in\mathbb{O}^3_0;\pm{\rm tr}(j(v))>0\right\},\quad
D_{0}:=\left\{\tilde{v}\in\mathbb{O}^3_0;{\rm tr}(j(v))=0\right\}.
\end{aligned}
\end{equation*}
$H^2_\mathbb O:=D_-/\sim\subset P\mathbb O_0^3$ is the {\it octonionic  hyperbolic space.}
Recall that ${\rm tr} j(v)={\rm tr} (\tilde v\tilde v^*D_1)=\tilde v^*D_1\tilde v,$ which  is the analog of  the Hermitian form in quaternionic case. We define a bilinear form on $\mathbb{O}^2$ (cf. \cite[p. 87]{Parker}) by
\begin{align}\label{nj}\langle j(v),j(w)\rangle=\frac12{\rm Re}\ {\rm tr}\left( j(v)j(w)+j(w)j(v)\right).\end{align}
Then  we have $\langle j(v),j(w)\rangle=\left|\tilde v^*D_1\tilde w\right|^2.$
Define \begin{align*}
(v,w):=\frac{|\tilde{v}^*D_{1}\tilde{w}|}
{|\tilde{v}^*D_{1}\tilde{v}|^{\frac{1}{2}}
|\tilde{w}^*D_{1} \tilde{w}|^{\frac{1}{2}}}.
\end{align*} The metric on $H_\mathbb{O}^2$ is given by (cf. \cite[p. 88]{Parker})
\begin{equation}\label{hyperdis}
\hbox{d}s^2=-4\frac{\left|\tilde{v}^*D_{1}\tilde{v}\right|
\left|\hbox{d}\tilde{v}^*D_{1}\hbox{d}\tilde{v}\right|-\left|\tilde{v}^*D_{1}
\hbox{d}\tilde{v}\right|^2}{\left|\tilde{v}^*D_{1}\tilde{v}\right|^2},
\end{equation}
at point  ${v}\in D_-$ and the distance $d(\cdot,\cdot)$ is given by
\begin{equation}\label{dis sie}
\cosh\left(\frac{d({v},{w})}{2}\right)=(v,w).
\end{equation}
$D_-$ is exactly the {\it octonionic Siegel domain}: $$\mathcal U=\{(x,y)\in\mathbb{O}^{2}: 2{\rm{Re}}\ y+|x|^{2}<0\}.$$

We introduce the positive definite form
$\langle v,w\rangle={v}_{1}\bar{w}_{1}+{v}_{2}\bar{w}_{2}$  on $\mathbb O^2$
and  the ball model for {octonionic hyperbolic space}
$
B^{16}=\left\{v\in\mathbb{O}^{2};\langle v,v\rangle<1\right\}.
$
The \emph{Cayley transform} is the map from the sphere $S^{15}$ minus the southern point to the boundary of the  Siegel domain $$\partial\mathcal U=\left\{(x,y)\in\mathbb{O}^{2}: 2{\rm{Re}}\ y+|x|^{2}=0\right\}$$ defined by
\begin{align*}
C:S^{15}\longrightarrow\mathcal U, \quad(v_{1},v_{2})\longmapsto\left(\sqrt{2}(1+v_{2})^{-1}v_{1},-(1-v_{2})(1+v_{2})^{-1}
\right).
\end{align*}

\subsection{The octonionic Heisenberg group}

The octonionic Heisenberg group $\mathscr{H}$ is $\mathbb{O}\oplus{\rm Im }\ \mathbb{O}$ equipped with the multiplication given by
\begin{align}\label{heiO}
(x,t)\cdot(y,s)=(x+y,t+s+2{\rm{Im}} (x\bar{y})),
\end{align}
where $(x,t),(y,s)\in\mathbb{O}\oplus{\rm Im }\ \mathbb{O}.$ Note that
\begin{equation}\begin{aligned}\label{2.17'''}
x\bar y=
&x_ay_a-x_0y_\gamma e_\gamma+y_0x_\gamma e_\gamma-x_\alpha y_\gamma\epsilon_{\alpha\gamma\beta}e_\beta,
\end{aligned}\end{equation}
by  (\ref{2.1'}), i.e. ${\rm Im}(x\bar y)=\left(-x_0y_\gamma+y_0x_\gamma-x_\alpha y_\beta\epsilon_{\alpha\beta\gamma}\right)e_\gamma.$  Therefore the multiplication of the octonionic Heisenberg group in terms of  real variables can be written as
\begin{align*}
(x,{t})\cdot({y},s)=\left(x+y,t_{\beta}+s_{\beta}+2
E_{ab}^{\beta}x_{a}y_{b}\right),
\end{align*}
where $x=(x_{0},\cdots,x_{7})\in\mathbb{R}^{8},$ $t=(t_{1},\cdots,t_{7})\in\mathbb{R}^{7},$ and $E^{\beta}$ are given by (\ref{o2.12}).

Since the definition of our  octonionic Heisenberg group is a bit different from that in \cite{Wang H} (with ${\rm{Im}} \left(\bar{x}y\right)$ replaced by ${\rm{Im}} \left( x\bar y\right)$), so are $E^\beta$'s.   The norm of the octonionic Heisenberg group $\mathscr{H}$ is defined by
\begin{align}\label{124}
\|(x,{t})\|:=(|x|^{4}+|{t}|^{2})^{\frac{1}{4}}.
\end{align}
By definition,
\begin{align}\label{2.14}
X_{a}=\frac{\partial}{\partial x_{a}}+2E_{ba}^{\beta}x_{b}
\frac{\partial}{\partial t_{\beta}},\quad a=0,1,\cdots,7,
\end{align}
are the left invariant vector fields on $\mathscr{H}.$
The standard $\mathbb{R}^{7}$-valued contact form of the group is
\begin{align}\label{0}
 \Theta_{0}:=\hbox{d}t-x\cdot\hbox{d} \bar{x} +\hbox{d}x\cdot \bar{x}.
\end{align}
If we write $\Theta_{0}=(\theta_{0;1},\cdots,
\theta_{0;7}),$ then we have
\begin{align}\label{215}
\theta_{0;\beta}=\hbox{d}t_{\beta}-2E_{ba}^{\beta}
x_{b}\hbox{d}x_{a},
\end{align}
by using (\ref{2.17'''}).
The standard Carnot-Carath\'eodory metric on the group is
$
g_{0}(X_{a},X_{b})=\delta_{ab}.
$
The transformations $I_\beta$ on $H_{0}$ are given by
$
I_{\beta}X_{a}=E_{ba}^{\beta}X_{b}.
$ Let $\nabla$ be the Biquard connection associated to this standard OC structure on $\mathscr H $, which is the flat model of OC manifolds, i.e.
\begin{align}\label{flat}
\nabla_{X_a}X_b=0,
\end{align} for any $X_a,X_b$   in (\ref{2.14}) and its scalar curvature and torsion  are  identically zero. The SubLaplacian on $\mathscr H$ is
$
\Delta_{0}=-\sum_{a=0}^{7}X_{a}^2.
$
We can identify $\mathscr{H}$ with the boundary of the   {Siegel domain},
by using the projection
\begin{equation}
\begin{aligned}\label{sth}
{\pi}:\qquad\partial\mathcal U\quad \longrightarrow\quad\mathscr{H},\qquad
(y,z) \longmapsto\left(\frac{y}{\sqrt{2}},\frac{|y|^{2}}{2}+z\right)=(x,t).
\end{aligned}\end{equation}

\subsection{The exceptional group ${\rm F}_{4(-20)}$} It is the group of  isometries of octionionic hyperbolic space $H_\mathbb O^2.$  The group is generated by
\begin{equation}\label{a1}
T=\begin{pmatrix}1&-1&-\frac{1}{2}\\0&1&1\\0&0&1\end{pmatrix},\quad D_\delta=\begin{pmatrix}\delta&0&0\\0&1&0\\0&0&\frac{1}{\delta}\end{pmatrix}, \quad R=\begin{pmatrix}0&0&1\\0&-1&0\\1&0&0\end{pmatrix},
\end{equation}
where $\delta\neq0\in\mathbb{R},$ and
\begin{equation}\label{a2}
S_\mu:\begin{pmatrix}y\\x\\1 \end{pmatrix}=\begin{pmatrix}\mu y\bar{\mu}\\ {\mu}x\\1 \end{pmatrix}
\end{equation}
for  unit imaginary octonion $\mu$ (cf.  \cite[p. 88]{Parker}).
The group generated by transformations $\{S_\mu:\mu\in{\rm Im}\ \mathbb O\}$ is the compact group ${\rm Spin} (7).$ We remark that in general $S_\mu\circ S_\nu\neq S_{\mu\nu}$ for unit imaginary octonions $\mu,\nu$. $F_{4(-20)}$ acts on $ \mathbb{O}^3_0$ as left action of matrices (\ref{a1})-(\ref{a2}),  and the induced action of $F_{4(-20)}$ on $\mathbb{O}^2 \subset P\mathbb{O}^3_0$ is \begin{align*}\gamma(  v)=\left(
\gamma({ \widetilde{ v}})_{1}\gamma(\tilde{  v})_{3}^{-1},
\gamma({  \widetilde{v}})_{2}\gamma({ \widetilde{v}})_{3}^{-1},1\right)^t.
\end{align*}
Then we have the following transformations   of $\mathscr{H}$:\\
(1) \emph{Dilations}: for given  positive number $\delta$ define
\begin{align}\label{dilation}
D_{\delta}:(x,t)\longrightarrow(\delta x,\delta^{2}t),\ \delta>0;
\end{align}
(2) \emph{Left translations}: for given $(y,s)\in \mathscr H$ define
\begin{align}\label{2.25}
\tau_{(y,s)}(x,t)=(y,s)\cdot(x,t);
\end{align}
(3) \emph{Rotations}: for given unit imaginary octonion $\mu$ define
\begin{align}\label{33}
S_\mu(x,t)=({\mu}x,\mu t \bar{\mu});
\end{align}
(4) \emph{Inversion}:
\begin{align}\label{44}
R(x,t)=\left(-\left(|x|^2-t\right)^{-1}x,
\frac{-t}{|x|^{4}+|t|^2}\right).
\end{align}

Recall that for a nilpotent Lie  group $G$  with Lie algebra $\mathfrak g,$ a continuous map $f:\Omega\rightarrow G$ is \emph{P-differentiable} at $p$ if the limit
$$Df(p)(q)=\lim_{\delta\rightarrow0^+} D_\delta^{-1}\circ L_{f(p)}^{-1}\circ f\circ L_p\circ D_\delta(q)$$
exists, uniformly for $q$ in compact subsets of $G$, where $L_p$ is a left translation by $p$ in $G$; if $Df(p)$ exists, then it is a strata-preserving homomorphism of $G.$ Then there is a Lie algebra homomorphism $\hbox{d}f(p):\mathfrak g\rightarrow\mathfrak g$ such that $Df(p)\circ exp=exp\circ \hbox{d}f(p).$ We call $Df(p)$ the {\it P-derivative} and $\hbox{d}f(p)$ the {\it P-differential} of $f$ at $p.$

We have the following version of Liouville-type theorem in the case of the  octonionic Heisenberg group. Denote the Carnot-Carath\'eodory distance
$d_{cc}(p,q):=\inf_{\gamma}\int_{0}^{1}|\gamma'(t)|\hbox{d}t$ for any $p,q\in \mathscr H,$  where $\gamma:[0,1]\rightarrow \mathscr H$ is taken over all  Lipschitzian horizontal curves, i.e. $\gamma'(t)\in H_{\gamma(t)}$ almost everywhere. Define balls $B_{cc}(x,r):=\left\{y\in\mathscr H:d_{cc}(x,y)<r\right\}.$
\begin{thm}\label{liou}{\rm(OC Liouville type theorem)}
Every conformal  contact transformation between open subsets of $\mathscr H$ is the restriction of the action of an element of ${\rm F}_{4(-20)}.$ Here conformal mapping is in the sense of sub-Riemannian manifold, i.e. $f^*g_0=\phi^2 g_0$ for some bounded  positive smooth function $\phi.$
\end{thm}
\begin{proof}Since $f^*g_0=\phi^2 g_0,$ it is direct to see that $f$ is locally  quasiconformal, i.e. we can write $f:\Omega\rightarrow\mathscr H$ for some domain $\Omega\subset\mathscr H$, and for any $y\in\Omega$ there exists a constant $k,r_0>0$ such that
\begin{align*}
B_{cc}\left(f(x),r/k\right)\subset f\left( B_{cc}\left(x,r\right)\right)\subset B_{cc}\left(f(x),kr\right),
\end{align*}
for some $r<r_0,x\in B_{cc}(y,r_0)\subset\Omega.$ It follows from Pansu's well-known rigidity theorem \cite[Corollary 11.2]{Pansu}, the $P$-differential of $f$ must be a similarity. Cowling and Ottazzi proved that (\cite[Theorem 4.1]{Cowling}): let $G$ be a Carnot group,  $\Omega$ be a connected open subset of $G,$ and let $f : \Omega\rightarrow G$ be a conformal mapping in the sense that $P$-differential is a similarity. Then if $G$ is the Iwasawa $N$ group of a real-rank-one simple Lie group, $f$ is the restriction to $\Omega$ of the action of an element of this associate Lie group on $N\cup\{\infty\}$. Thus $f$ extends analytically to a conformal map on $G$ or $G\setminus\{p\}$ for some point $p.$ Since the  octonionic Heisenberg group is an Iwasawa $N$ group of ${\rm F}_{4(-20)},$ the theorem follows directly.
\end{proof}
\begin{prop}\label{o2.9}
For  $\gamma\in {\rm F}_{4(-20)}$  we have
\begin{align}\label{pullg0}
\gamma^{*}g_{0}=\phi^2 g_{0},
\end{align}
for some positive smooth function $\phi.$  In particular, $\phi(x,t)=\delta,1,1$ and $\left(|x|^4+|t|^2\right)^{-\frac12}$ for the dilation $D_\delta,$  left translation $\tau_{(y,s)},$  rotation $S_\mu $ and the inversion $R,$ respectively.
\end{prop}
\begin{proof}
The definition of similarity implies that $g_0\left(f_*X_a,f_*X_a\right)=\phi^2(x,t)g(X_a,X_a)$ for any $X_a$ given in (\ref{2.14}). So  (\ref{pullg0}) follows. It is obvious  that the proposition holds for dilations $D_\delta,$  left translation $\tau_{(y,s)}$ and rotation $S_\mu.$ Here we only need to  prove (\ref{pullg0}) for the inversion $R$.

Since given a point $(x,t)\in\mathscr{H}$, we can choose suitable $D_\delta,\tau_{(y,s)}$ and $S_\mu$ such that $S_\mu\circ D_\delta\circ\tau_{(y,s)}(x,t)=(1,0,\cdots,0, t')$  for some $t'\in {\rm Im}\,\mathbb{O}$,
 it is sufficient to prove the result  at point $(1,0,\cdots,0, t)$.
For $(x,t)=(1,0,\cdots,0,t ),$ we have
\begin{equation}\begin{aligned}\label{kappa}
\left(R_*X_1\right)&F(R(x,t))=\left.\frac{\hbox{d}}{\hbox{d}\kappa}F\left(R\left[ (1,0,\cdots,0,t)\cdot(\kappa,0,\cdots,0)\right]\right)\right|_{\kappa=0}\\ =&\left.\frac{\hbox{d}}{\hbox{d}\kappa}F\left(R (1+\kappa,0,\cdots,0,t)\right)\right|_{\kappa=0} =\left.\frac{\hbox{d}}{\hbox{d}\kappa} F\left(-\frac{\mathcal{E}_0}{\rho},
\cdots,-\frac{\mathcal{E}_7}{\rho},\frac{-t}{\rho}\right) \right|_{\kappa=0},
\end{aligned}\end{equation}
for any smooth function $F$ with
\begin{align*}
\mathcal{E}_0=(1+\kappa)^3,\quad\mathcal{E}_\beta:=t_\beta(1+\kappa),\quad \rho=(1+\kappa)^4+|t|^2,\quad\ \beta=1,\cdots,7.
\end{align*}
Then the right hand side of (\ref{kappa}) equals to
\begin{equation}\begin{aligned}\label{idr}
\frac{1-3|t|^2}{(1+|t|^2)^2}F_0+\frac{t_\beta\left(3-|t|^2\right)}{(1+|t|^2)^2}F_\beta+ 4\frac{t_\beta}{(1+|t|^2)^2}F_{7+\beta}\\=A_0^l\left(F_l+2
E_{kl}^{\beta}\left(-\frac{\mathcal E_k|_{\kappa=0}}{1+|t|^2}\right)F_{7+\beta}\right)(R(x,t)),
\end{aligned}\end{equation}
where
\begin{equation*}\begin{aligned}
A_0^0=\frac{1-3|t|^2}{(1+|t|^2)^2},\quad
A_0^\beta=\frac{t_\beta\left(3-|t|^2\right)}{(1+|t|^2)^2}.
\end{aligned}\end{equation*}
Since
\begin{equation*}\begin{aligned}
-A_0^lE_{kl}^\beta\frac{\mathcal E_k|_{\kappa=0}}{1+|t|^2}=&-A_0^lE_{0l}^\beta\frac{1}{1+|t|^2}-A_0^lE_{\alpha l}^\beta\frac{t_\alpha}{1+|t|^2}\\=&-A_0^0E_{00}^\beta\frac{1}{1+|t|^2} -A_0^\alpha E_{0\alpha}^\beta\frac{1}{1+|t|^2}-A_0^0E_{\alpha 0}^\beta\frac{t_\alpha}{1+|t|^2}-A_0^\gamma E_{\alpha \gamma}^\beta\frac{t_\alpha}{1+|t|^2}\\=&\frac{(|t|^2-3)t_\alpha}{(1+|t|^2)^3} E_{0\alpha}^\beta +\frac{(1-3|t|^2)t_\alpha}{(1+|t|^2)^3}E_{0\alpha}^\beta +\frac{(|t|^2-3)}{(1+|t|^2)^3} E_{\alpha\gamma}^\beta t_\alpha t_\gamma\\=&\frac{-2}{(1+|t|^2)^2}E_{0\alpha}^\beta t_\alpha=\frac{2t_\beta}{(1+|t|^2)^2}.
\end{aligned}\end{equation*}
The third identity holds by $E_{00}^\beta=0$ and the fourth identity holds by $E_{\alpha\gamma}^\beta t_\alpha t_\gamma=0.$
Thus $\left.R_*X_1\right|_{R(1,0,\cdots,0,t)}=\left.A_0^lX_l\right|_{R(1,0,\cdots,0,t)}.$ We can easily check $\sum_{l=0}^7\left(A_0^l\right)^2=\frac{1}{1+|t|^2}.$ The proposition is proved.
\end{proof}
As mentioned before, an OC manifold $(M,g,\mathbb{I})$ is always \emph{spherical}, i.e. it is locally conformally  OC equivalent to an open set of  the octionionic  Heisenberg group with standard OC structure. A conformal class can be described topologically as a manifold   whose  coordinate charts are given by open subsets of the octonionic  Heisenberg group and elements of ${\rm F}_{4(-20)}$ as transition maps. So we can omit $\mathbb I$ in the notion $(M,g,\mathbb I)$ of an OC manifold.
\section{The  OC Yamabe operator and its transformation formula under conformal transformations}
\subsection{The Biquard connection}\qquad

{\bf Theorem} {\rm(\cite[Theorem B]{biquard})}
For an OC manifold  with  Carnot-Carath\'eodory metric $g$ on $H$, there exists a unique connection $\nabla$ on $H$ and a unique supplementary subspace $V$ of $H$ in
$TM$, such that\\
(i) $\nabla$ preserves the decomposition  $H\oplus V$ and the metric;\\
(ii) for $X,Y\in H$, one has $T_{X,Y}=-[X,Y]_{V}$;\\
(iii) $\nabla$ preserves the $ {\rm Spin}(7)$-structure on $H$;\\
(iv) for $R\in V$, the endomorphism $\cdot\rightarrow(T_{R,\cdot})_{H}$ of $H$
 lies in  $\mathfrak{so}_{7}^\perp$;\\
(v) the connection on $V$ is induced by the natural identification of $V$
with the subspace $\mathbb{R}^7$ of the endomorphisms of $H$.
\vskip 5mm
Since the Biquard  connection preserving Carnot-Carath\'eodory metric on it, we have $$\nabla(\hbox{d}\theta_\alpha)\in\Lambda^1 H\otimes\mathfrak{so}(8)=\Lambda^1H\otimes\left(\mathfrak{so}(7)\oplus\mathbb{R}^7
\right),$$ and that the connection preserves the {\it ${\rm Spin}(7)$ structure} if its component in $\Lambda^1H\otimes\mathbb{R}^7$ vanish (cf. \cite[p. 84]{biquard}).
It satisfies the following properties proved by Biquard \cite[Proposition II.1.7, II.1.9]{biquard}. Recall that {\it Reeb vector fields} $R_\alpha,\alpha\in\mathbb Z_{8+},$ satisfy $i_{{R}_\alpha}\hbox{d}{\theta}_\alpha|_H=0,i_{{R}_\alpha}\hbox{d}{\theta}_\beta|_H
=-i_{{R}_\beta}\hbox{d}{\theta}_\alpha|_H$ for $\alpha\neq\beta$ and ${\theta}_\beta({R}_\alpha)=\delta_{\alpha\beta}.$
\begin{prop}\label{2.1.7}{\rm(1)} The Biquard  connection satisfies
\begin{align*}
\nabla{\rm d}\theta_\alpha=-(i_{R_\beta}{\rm d}\theta_\alpha)|_H\otimes{\rm d}\theta_\beta,
\end{align*}
and in particular, $(i_{R_\beta}{\rm d}\theta_\alpha)|_H=-(i_{R_\alpha}{\rm d}\theta_\beta)|_H$ and $(i_{R_\alpha}{\rm d}\theta_\alpha)|_H=0.$ By isomorphism ${\rm d}\theta_\alpha\rightarrow R_\alpha,$ we have \begin{align}\label{nablaR}
\nabla R_\alpha=-(i_{R_\beta}{\rm d}\theta_\alpha)|_H\otimes
R_\beta.
\end{align}
{\rm(2)} It is a metric connection on $V$ and
\begin{align}\label{xr}
\nabla_XR=[X,R]_V,\quad X\in H,R\in V.
\end{align}
\end{prop}
\begin{prop}{\rm(1)}\label{II}
$\mathfrak{so}(7)\cong \mathop{\rm span}\{I_\alpha I_\beta; 1\leq \alpha<\beta\leq7 \};$\\
{\rm(2)}\label{module}
$\mathbb R^7\cong{\rm span}\{I_\alpha\}$ is an $\mathfrak{so}(7)$ module.
\end{prop}
\begin{proof}(1)
Recall that  $\left(N^{\alpha\beta}\right)_{\gamma\delta}=2\epsilon_{\alpha\beta\delta\gamma},$  where $\epsilon_{\alpha\beta\delta\gamma}$ is given by  (\ref{mnps}). $\left\{N^{\alpha\beta};1\leq \alpha<\beta\leq 7\right\}$ are linearly independent, because $N^{\alpha\beta}$ is $(8\times8)$-antisymmetric matrix with only $4$ entries  not zero, and for any fixed $\alpha<\beta,$ the nonzero elements are at different entries (cf. Proposition \ref{abc}). As $I_\alpha$ is an $(8\times8)$-antisymmetric matrix with $(0,\alpha)$ entry to be  $-1,$ while $(0,a)$ entry  of  $N^{\alpha\beta}$ is always $0$ for any $a,$ so  $\left\{I_\alpha;\alpha=1,\cdots,7\right\}\cup\left\{N^{\alpha\beta};1\leq \alpha<\beta\leq 7\right\}$ are linearly independent.  $\left\{I_\alpha I_\beta;1\leq \alpha<\beta\leq 7\right\}$ is closed under Lie brackets. Since
\begin{align*}
\left[I_\alpha I_\beta,I_\gamma I_\delta\right]=I_\alpha I_\beta I_\gamma I_\delta-I_\gamma I_\delta I_\alpha I_\beta=0,
\end{align*}
for $\alpha,\beta,\gamma,\delta$ different, and
\begin{align*}
\left[I_\alpha I_\beta,I_\alpha I_\gamma\right]=I_\alpha I_\beta I_\alpha I_\gamma-I_\alpha I_\gamma I_\alpha I_\beta=2I_\beta I_\gamma,
\end{align*}
for $\alpha,\beta,\gamma$ different.
Moreover, $\left\{I_\alpha I_\beta; 1\leq \alpha<\beta\leq 7\right\}$ are linearly independent by $I_\alpha I_\beta=I_\gamma-N^{\alpha\beta}$ if $e_\alpha e_\beta=e_\gamma.$ Therefore,
$\left\{I_\alpha;\alpha=1,\cdots,7\right\}\cup\left\{I_\alpha I_\beta; 1\leq \alpha<\beta\leq 7\right\}$ are linearly independent. Note that $\mathbb R^7={\rm span}\left\{I_\alpha;\alpha=1,\cdots,7\right\},$   $\mathfrak{so}(8)\cong\mathbb R^7\oplus\mathfrak{so}(7),$ and dim $\mathfrak{so}(7)=21=$dim $\left\{I_\alpha I_\beta; 1\leq \alpha<\beta\leq 7\right\}.$ Then we have $\mathfrak{so}(7)\cong{\rm span}\{I_\alpha I_\beta ;1\leq \alpha<\beta\leq7 \}.$\\
(2)
We have   $[I_\alpha I_\beta,I_\gamma]=I_\alpha I_\beta I_\gamma-I_\gamma I_\alpha I_\beta=0,$ for $\alpha,\beta,\gamma$ different and $[I_\alpha I_\beta,I_\alpha]=I_\alpha I_\beta I_\alpha-I_\alpha I_\alpha I_\beta=2I_\beta\in\mathbb R^7.$ The proposition is proved.
\end{proof}
\begin{prop}
The connection coefficients of the Biquard connection is $\mathfrak{so}(7)$-valued.
\end{prop}
\begin{proof}
We write the connection coefficients as ${\Gamma_{ab}}^c,$ i.e. Proposition \ref{2.1.7} (1) implies that $\nabla_{V_a}V_b={\Gamma_{ab}}^cV_c$ for a local  frame $\{V_a\}.$  Since $\hbox{d}\theta_\alpha(\cdot,\cdot)=g(I_\alpha\cdot,\cdot),$ we have
\begin{align}\label{nablaI}\nabla I_\alpha|_H=-(i_{R_\beta}{\rm d}\theta_\alpha)|_H\otimes I_\beta,
\end{align}
which is equivalent to \begin{align*}
{\Gamma_{ad}}^c{I_{\alpha;b}}^d-{\Gamma_{ab}}^d{I_{\alpha;d}}^c=
-\hbox{d}\theta_\alpha(R_\beta,V_a){I_{\beta;b}}^c,
\end{align*}
i.e. $[\Gamma_a,I_\alpha]\in\mathbb{R}^7.$
On the other hand $[\Gamma_a,I_\alpha]=\left[(\Gamma_a)_{\mathfrak{so}(7)},I_\alpha\right]+
\left[(\Gamma_a)_{\mathbb{R}^7},I_\alpha\right],$
and  $\left[(\Gamma_a)_{\mathfrak{so}(7)},I_\alpha\right]\in\mathbb{R}^7$ by Proposition \ref{module}. But $\left[(\Gamma_a)_{\mathbb{R}^7},I_\alpha\right]\notin\mathbb{R}^7$ by (\ref{Nij}) if $\left(\Gamma_a\right)_{\mathbb R^7}\neq0$. So we must have $\left(\Gamma_a\right)_{\mathbb R^7}=0,$ i.e.  $\Gamma_a\in\mathfrak{so}(7).$
\end{proof}

The curvature of Biquard  connection is defined by $R(X,Y)=\nabla_X\nabla_Y-\nabla_Y\nabla_X-\nabla_{[X,Y]}$ and Ricci curvature is defined by $Ric(X,Y)= g(R(V_a,X)Y,V_a)$ for any $X,Y\in H,$ where $V_a$ is  local orthonormal basis of horizontal subspace $H.$ The scalar curvature is $s_g={\rm tr}^H Ric.$
\subsection{The Biquard connection under conformal transformations}
 When the octonionic structure $\mathbb I$ is rotated by ${\rm SO}(7),$ the Carnot-Carath\'eodory metric also satisfies (\ref{2.01})-(\ref{2,1}) for $1$-form $\Theta$ rotated. Hence, the Reeb vectors are also rotated by definition, and the Biquard connection is the same. So when consider conformal transformations, we can fixed the octonionic structure $\mathbb I.$
\begin{prop}\label{confch}
Under  the conformal change $\tilde{g}=f^{2}g_0$ on the octonionic Heisenberg group, the scalar curvature becomes \begin{align}\label{sgg}
s_{\tilde{g}}=-f^{-2}\left((2Q-2){\rm tr}^{H}(\nabla K)+{(Q-1)(Q-2)}|K|^{2}\right)=-f^{-2}\left(42{\rm tr}^{H}(\nabla K)+420|K|^{2}\right),
\end{align}
where $K:=f^{-1}{\rm d}f.$
\end{prop}
To prove this proposition we need to know the transformation formulae of Reeb vector fields and the Biquard connection under the OC conformal transformation. Recall that the wedge product of 1-forms  $\phi$ and  $\psi$ is given by \begin{align}\label{345}
(\phi\wedge\psi)(X,Y):
=\phi(X)\psi(Y)-
\phi(Y)\psi(X),
  \end{align}
for any vector field $X$ and $Y.$ Then we have $\hbox{d}\phi(X,Y)=X\phi(Y)-Y\phi(X)-\phi([X,Y]).$
For $X,Y\in H,$ define  $X\wedge Y$ as the endmorphism of $H$ by $$X\wedge Y(Z):=g(X,Z)Y-g(Y,Z)X$$ for any $Z\in H.$

\begin{prop}\label{H}
The Reeb fields of $\tilde{g}=f^2g$ associate with $\tilde\theta_\alpha=f^2\theta_\alpha,\alpha=1,\cdots,7$ are the vectors
\begin{align}\label{2.17'}
\widetilde{R}_\alpha=f^{-2}(R_\alpha+r_\alpha),\quad r_\alpha=-2I_\alpha K^\sharp,
\end{align}
where the  vector $K^\sharp\in H$ is defined by $K(X):=\langle K^\sharp,X\rangle_g$ for any $X\in H.$ The Biquard connection of $\tilde g$ satisfies
\begin{equation}\begin{aligned}
\widetilde{\nabla}_X=\nabla_X+A_X,\quad {\rm for}\ X\in H,
\end{aligned}\end{equation}
with
\begin{align}\label{2.19'}
A_X=K(X)+U_X:=K(X)+\langle I_\alpha K^\sharp,X\rangle I_\alpha+K^\sharp\wedge X+I_\alpha K^\sharp\wedge I_\alpha X.
\end{align}
\end{prop}
\begin{proof}
By Proposition \ref{2.1.7}, the vector field $\widetilde R_\alpha$ is characterized by $i_{\widetilde{R}_\alpha}\hbox{d}\tilde{\theta}_\alpha|_H=0,$ $i_{\widetilde{R}_\alpha}\hbox{d}\tilde{\theta}_\beta|_H= -i_{\widetilde{R}_\beta}\hbox{d}\tilde{\theta}_\alpha|_H$ and $\tilde{\theta}_\beta(\tilde{R}_\alpha)=\delta_{\alpha\beta}.$ Since
$ f^{-2}\hbox{d}\tilde{\theta}_\alpha=\hbox{d}{\theta}_\alpha+2K\wedge\theta_\alpha, $
we have (\ref{2.17'}) immediately. As $\widetilde T_{X,Y}=-[X,Y]_{\widetilde V}=\hbox{d}{\tilde \theta}_\alpha(X,Y)\widetilde R_\alpha,$ for $X,Y\in H,$ we have
\begin{equation}\begin{aligned}\label{2.21'}
g(A_XY,Z)-g(A_YX,Z)= g\left(\widetilde{T}_{X,Y}-T_{X,Y},Z\right) = g\left(
\hbox{d}{\theta}_\alpha(X,Y){r}_\alpha,Z\right)=2\hbox{d}{\theta}_\alpha
(X,Y)K(I_\alpha Z),
\end{aligned}\end{equation}
while $\widetilde{\nabla}\widetilde{g}=0$ yields
\begin{equation*}\begin{aligned}
0=\widetilde{\nabla}_X\widetilde{g}(Y,Z)=&X(\widetilde{g}(Y,Z))-f^2g(\nabla_XY+A_XY,Z)
-f^2g(Y,\nabla_XZ+A_XZ)\\=&2f^2K(X)g(Y,Z)-f^2g(A_XY,Z)-f^2g(Y,A_XZ),
\end{aligned}\end{equation*}
i.e.
\begin{align}\label{3.7'}
g(A_XY,Z)+g(Y,A_XZ)=2K(X)g(Y,Z).
\end{align}
Alternating  $X,Y,Z,$ we have
\begin{equation}\begin{aligned}\label{3.8'}
g(A_YZ,X)+g(Z,A_YX)=2K(Y)g(Z,X),\\
g(A_ZX,Y)+g(X,A_ZY)=2K(Z)g(X,Y).
\end{aligned}\end{equation}
Take the sum of  the first two equations in (\ref{3.7'})-(\ref{3.8'})  and then minus the last one to get
\begin{equation}\begin{aligned}\label{2.25''}
g(A_XY,Z)+g(Z,A_YX)=&-2\hbox{d}{\theta}_\alpha
(X,Z)K(I_\alpha Y)-2\hbox{d}{\theta}_\alpha
(Y,Z)K(I_\alpha X)\\&+2K(X)g(Y,Z)+2K(Y)g(Z,X)-2K(Z)g(X,Y)
\end{aligned}\end{equation}
by using (\ref{2.21'}).
Then (\ref{2.19'}) is the sum of (\ref{2.21'}) and  (\ref{2.25''}).
\end{proof}
Denote by $D_{\mathfrak{so}(7)}$ and $D_{\mathbb R^7}$  the projection to $\mathfrak{so}(7)$ and $\mathbb R^7$ for $D\in\mathfrak{so}(8)$ in the decomposition  (\ref{so7}) with respect to the Killing form  $\langle\cdot,\cdot\rangle$ on $\mathfrak{so}(8):$ $\langle D,F\rangle=tr(F^*D)=\frac18\sum_{a=0}^7\langle D(V_a), F(V_a)\rangle$ for $D, F\in \mathfrak{so}(8),$ where $\{V_a\}_{a=0}^7$ is a local orthonormal basis of  $H.$ By (\ref{so7}), we have  $$D_{\mathfrak{so}(7)}=D-D_{\mathbb{R}^7},$$ for any $1$-form $D$ with value in $\mathfrak{so}(8),$ where $D_{\mathbb{R}^7}=\sum_{\alpha=1}^7D_\alpha I_\alpha$ with \begin{align}\label{r7}D_\alpha=\frac{1}{8}\sum_a\langle DV_a,I_\alpha V_a\rangle.\end{align}
\begin{prop}\label{V}
The Biquard connection $\widetilde{\nabla}$ satisfies \begin{equation}\begin{aligned}\label{2.22'}
\widetilde{\nabla}_{R_\alpha}=&\nabla_{R_\alpha}+A_{R_\alpha}:=\nabla_{R_\alpha}+
K(R_\alpha)+(U_{R_\alpha})_{\mathfrak{so}(7)},\\
\widetilde{T}_{R_\alpha+r_\alpha,X}=&{T}_{R_\alpha,X}+K(R_\alpha)X-2|K|^2I_\alpha X-
(U_{R_\alpha})_{\mathbb{R}^7}X,
\end{aligned}\end{equation}
where
\begin{align}\label{ur}
U_{R_\alpha}X=2I_\alpha\nabla_XK^\sharp+4\sum_{\beta=1}^7\langle I_\beta K^\sharp,X\rangle I_\alpha I_\beta K^\sharp-4\sum_{\beta\neq \alpha}\langle I_\beta I_\alpha K^\sharp,X\rangle I_\beta K^\sharp.
\end{align}
\end{prop}
\begin{proof}
Let us write
\begin{equation}
\widetilde{\nabla}_B=
\nabla_B+K(B)+\mathscr A(B)
\end{equation}
for the connection acting on the horizontal subbundle for  $B\in V$  and some $\mathscr A\in\Lambda^1\otimes\mathfrak{gl}(8),$ where $\widetilde \nabla$ is the Biquard connection associated to $\widetilde g$ and $\widetilde V={\rm span} \{\widetilde R_\alpha\}_{\alpha=1}^7$ given by (\ref{2.17'}). This  connection preserves $f^2g $ with torsion   given by
\begin{equation}\begin{aligned}\label{2.25'}
\widetilde T_{R_\alpha+r_\alpha,X}=&\widetilde\nabla_{R_\alpha+r_\alpha}X-\widetilde \nabla_X({R_\alpha+r_\alpha})
-[R_\alpha+r_\alpha,X]\\=&\widetilde\nabla_{R_\alpha+r_\alpha}X -[R_\alpha+r_\alpha,X]_{H/\widetilde{V}}\\
=&\nabla_{R_\alpha}X+K(R_\alpha)X+\widetilde{\nabla}_{r_\alpha}X -[R_\alpha,X]_{H/\widetilde{V}}-[r_\alpha,X]_{H/\widetilde{V}}+\mathscr A(R_\alpha)X,
\end{aligned}\end{equation}
 by  (\ref{xr}) for $\widetilde \nabla$, where  $H/\widetilde{V}$ denotes the projection onto $H$ along  the direction of $\widetilde{V}.$ On the one hand,
\begin{align}
\widetilde{\nabla}_{r_\alpha}X-[r_\alpha,X]_{H/\widetilde{V}} =\widetilde{\nabla}_X{r_\alpha},
\end{align}
by the definition of $\widetilde{\nabla},$ and on the other hand
\begin{align}
\nabla_{R_\alpha}X-[R_\alpha,X]_{H/\widetilde{V}}=T_{R_\alpha,X}+[R_\alpha,X]_{H/{V}}-
[R_\alpha,X]_{H/\widetilde{V}}=T_{R_\alpha,X}-\hbox{d}\theta_\beta(R_\alpha,X)r_\beta.
\end{align}
Inserting these two identities into (\ref{2.25'}) to get \begin{align*}
\widetilde{T}_{R_\alpha+r_\alpha,X}=T_{R_\alpha,X}+K(R_\alpha)X +\widetilde{\nabla}_X{r_\alpha}-\hbox{d}\theta_\beta(R_\alpha,X)r_\beta+\mathscr A(R_\alpha)X.
\end{align*}
We can choose an orthonomal frame such that $\nabla I_\alpha(p)=0$ for any $\alpha$ at a fixed point $p,$ thus we have $\hbox{d}\theta_\alpha(R_\beta,X)=0$ for any $X\in H.$
To calculate $\widetilde\nabla_{X} r_\alpha,$ note that
\begin{equation*}\begin{aligned}
\widetilde{\nabla}_XI_\alpha=&i_{R_\alpha+r_\alpha}\left(
\hbox{d}\theta_\beta+2K\wedge\theta_\beta\right)_X  I_\beta=\left(
\hbox{d}\theta_\beta(r_\alpha,X)-2K(X)\delta_{\alpha\beta}\right) I_\beta=-2\sum_{\beta\neq \alpha}\langle I_\beta I_\alpha K^\sharp,X\rangle I_\beta,
\end{aligned}\end{equation*}
by  (\ref{nablaI}). By Proposition \ref{H}, we have
\begin{align*}
\widetilde{\nabla}_XK^\sharp={\nabla}_XK^\sharp+|K|^2X+2
\langle I_\beta K^\sharp,X\rangle I_\beta K^\sharp.
\end{align*}
Thus
\begin{align*}
\widetilde{\nabla}_X(I_\alpha K^\sharp)=I_\alpha{\nabla}_XK^\sharp+
|K|^2I_\alpha X+2\langle I_\beta K^\sharp,X\rangle I_\alpha I_\beta K^\sharp-
2\sum_{\beta\neq \alpha}\langle I_\beta I_\alpha K^\sharp,X\rangle I_\beta K^\sharp.
\end{align*}
Therefore
\begin{equation*}\begin{aligned}
\widetilde{T}_{R_\alpha+r_\alpha,X}=T_{R_\alpha,X}+K(R_\alpha)X-2|K|^2I_\alpha X-U_{R_\alpha}X+\mathscr A(R_\alpha)X,
\end{aligned}\end{equation*}
by $r_\alpha=-2I_\alpha K^\sharp,$ where $U_{R_\alpha}$ is defined in (\ref{ur}).
Since $\widetilde T_{\widetilde R_\alpha}(\cdot)$ is an the endomorphism of $H$ lying  in $\mathfrak{so}_7^\perp,$   (\ref{2.22'}) follows.
\end{proof}
 \subsection{Proof of Proposition \ref{confch}.}  Recall that $\{X_a\}_{a=0}^7$ is the standard orthonormal basis (\ref{2.14}) of $H_0$ on $\mathscr H$ and   $\nabla$ is flat, i.e. $\nabla_{X_a}X_b=0.$  Denote by  $s_{\tilde g}$ the scalar curvature of $f^2g_0$ and by $\widetilde X_a=f^{-1}X_a$ the local orthonormal basis of $\tilde g.$ By formulae in  Proposition \ref{H}, Proposition \ref{V} and direct calculation, we have
\begin{equation}\begin{aligned}\label{corf}
f^2{s}_{\widetilde g}=&f^2\left\langle \widetilde{R}(\widetilde{X}_a,\widetilde{X}_b)
\widetilde{X}_b,\widetilde{X}_a\right\rangle_{f^2g_0}=\left\langle \widetilde{R}(X_a,X_b)X_b,X_a\right\rangle_{g_0}\\
=&\left\langle \widetilde{\nabla}_{X_a}\widetilde{\nabla}_{X_b}X_b-
\widetilde{\nabla}_{X_b}\widetilde{\nabla}_{X_a}X_b-\widetilde{\nabla}_{[X_a,X_b]
}X_b,X_a\right\rangle\\=&\left\langle \widetilde{\nabla}_{X_a}\left(K(X_b)X_b+U_{X_b}X_b\right)-
\widetilde{\nabla}_{X_b}\left(K(X_a)X_b+U_{X_a}X_b\right)
-A_{[X_a,X_b]}X_b,X_a\right\rangle\\
=&\left\langle X_aK(X_b)X_b+{\nabla}_{X_a}
(U_{X_b}X_b)+U_{X_a}(U_{X_b}{X_b})-X_bK(X_a)X_b\right.\\
&-\left.{\nabla}_{X_b}(U_{X_a}X_b)-U_{X_b}(U_{X_a}{X_b})-A_{[X_a,X_b]}X_b,
X_a\right\rangle
\end{aligned}\end{equation}
where we used the flatness of $\nabla$ in the third and fourth identities. We write $\langle\cdot,\cdot\rangle$ instead of $\langle\cdot,\cdot\rangle_{g_0}$ for simplicity. Note that $\left\langle X_aK(X_b)X_b\right.$ $\left.-X_bK(X_a)X_b,X_a\right\rangle=0.$ As
\begin{equation*}\begin{aligned}
\left\langle {\nabla}_{X_a}(U_{X_b}X_b),
X_a\right\rangle=&2\left\langle I_\alpha\nabla_{X_a}K^\sharp,X_b\right\rangle\left\langle I_\alpha X_b,X_a\right\rangle+
\left\langle \nabla_{X_a}K^\sharp,X_b\right\rangle\left\langle X_b,X_a\right\rangle-8\left\langle \nabla_{X_a}K^\sharp,X_a\right\rangle\\=&-14\left\langle \nabla_{X_a}K^\sharp,X_a\right\rangle+\left\langle \nabla_{X_a}K^\sharp,X_a\right\rangle-8\left\langle \nabla_{X_a}K^\sharp,X_a\right\rangle=-21X_aK(X_a).
\end{aligned}\end{equation*}
Here and in the following we use $\nabla I_\alpha=0,\nabla X_a=0$ and $\langle X_a,X_b\rangle=\delta_{ab}$ repeatedly. We also have used the  property that  for any  fixed $a$ and $\alpha,$ there exists a  unique $b$ such that $\langle I_\alpha X_b,X_a\rangle$ nonvanishing and equal to $\pm1,$ i.e. $I_\alpha X_b=\pm X_a.$  For example, \begin{equation*}\begin{aligned}
\sum_{\alpha,\beta}\left\langle I_\alpha X_a,X_b\right\rangle\left\langle I_\beta X_b,X_a\right\rangle=-\sum_{\alpha,\beta}\left\langle  X_a,I_\alpha X_b\right\rangle\left\langle I_\beta X_b,X_a\right\rangle=-\sum_{\alpha}\left|\left\langle  X_a,I_\alpha X_b\right\rangle\right|^2=-56.
\end{aligned}\end{equation*}
Similarly, by $U_{X_a}$ being antisymmetric, we have
\begin{equation*}\begin{aligned}
&\left\langle U_{X_a}(U_{X_b}{X_b}),
X_a\right\rangle=-\left\langle U_{X_b}{X_b},U_{X_a}{X_a}\right\rangle\\=&-\left\langle2\left\langle I_\alpha K^\sharp,X_b\right\rangle I_\alpha X_b+K(X_b)X_b-8K^\sharp,2\left\langle I_\beta K^\sharp,X_a\right\rangle I_\beta X_a+K(X_a)X_a-8K^\sharp\right\rangle\\
=&-4\left\langle I_\alpha K^\sharp,X_b\right\rangle \left\langle I_\beta K^\sharp,X_a\right\rangle \left\langle I_\alpha X_b,I_\beta X_a \right\rangle -2\left\langle I_\alpha K^\sharp,X_b\right\rangle\left\langle I_\alpha X_b,X_a \right\rangle K(X_a)\\&+16\left\langle I_\alpha K^\sharp,X_b\right\rangle\left\langle I_\alpha X_b,K^\sharp\right\rangle -2\left\langle I_\beta K^\sharp,X_a\right\rangle\left\langle X_b, I_\beta X_a \right\rangle K(X_b)-K(X_b)K(X_a)\left\langle X_b,X_a\right\rangle\\&+8K(X_b)K(X_b)+16\left\langle K^\sharp,I_\beta X_a\right\rangle\left\langle I_\beta K^\sharp,X_a\right\rangle+8K(X_a)K(X_a)-64\left\langle K^\sharp,K^\sharp\right\rangle
\\=&(-196+14-112+14-1+8-112+8-64)|K|^2=-441|K|^2.
\end{aligned}\end{equation*}
Here $|K|^2=\sum_{a=0}^7\left|\left\langle K^\sharp,X_a\right\rangle\right|^2=\left\langle K^\sharp,K^\sharp\right\rangle.$ We also have
{\small\begin{equation*}\begin{aligned}
-\left\langle{\nabla}_{X_b}(U_{X_a}X_b),X_a\right\rangle=
&-\left\langle I_\alpha\nabla_{X_b}K^\sharp,X_a\right\rangle\left\langle I_\alpha X_b,X_a\right\rangle-
X_b\left(K(X_b)\right)\left\langle X_a,X_a\right\rangle+\left\langle X_a,X_b\right\rangle\left\langle \nabla_{X_b}K^\sharp,X_a\right\rangle\\&-\left\langle I_\alpha\nabla_{X_b}K^\sharp,X_b\right\rangle\left\langle I_\alpha X_a,X_a\right\rangle+\left\langle I_\alpha X_a,X_b\right\rangle\left\langle I_\alpha\nabla_{X_b}K^\sharp,X_a\right\rangle\\=&-7\left\langle \nabla_{X_a}K^\sharp,X_a\right\rangle-8X_b\left(K(X_b)\right)+
\left\langle \nabla_{X_a}K^\sharp,X_a\right\rangle-7\left\langle \nabla_{X_a}K^\sharp,X_a\right\rangle\\=&-21X_aK(X_a),
\end{aligned}\end{equation*}}
and
{\begin{equation*}\begin{aligned}
-&\left\langle U_{X_b}(U_{X_a}{X_b}),
X_a\right\rangle=\left\langle U_{X_a}{X_b},U_{X_b}{X_a}\right\rangle\\=&\left\langle\left\langle I_\alpha K^\sharp,X_a\right\rangle I_\alpha X_b+\left\langle K^\sharp,X_b\right\rangle X_a-\left\langle X_a,X_b\right\rangle K^\sharp+\left\langle I_\alpha K^\sharp,X_b\right\rangle I_\alpha X_a-\left\langle I_\alpha X_a,X_b\right\rangle I_\alpha K^\sharp,\right.\\&\left.\left\langle I_\beta K^\sharp,X_b\right\rangle I_\beta X_a+\left\langle K^\sharp,X_a\right\rangle X_b-\left\langle X_b,X_a\right\rangle K^\sharp+\left\langle I_\beta K^\sharp,X_a\right\rangle I_\beta X_b-\left\langle I_\beta X_b,X_a\right\rangle I_\beta K^\sharp
\right\rangle\\=&\left(-35+7+56+35+1-1+7-7+7-1+8+7\right.\\&\left.+56
+7+7-35-35+35-7-35-56\right)|K|^2=21|K|^2.
\end{aligned}\end{equation*}}
Here we have use for example
{\begin{equation*}\begin{aligned}
&\left\langle\left\langle I_\alpha K^\sharp,X_a\right\rangle I_\alpha X_b,\left\langle I_\beta K^\sharp,X_b\right\rangle I_\beta X_a
\right\rangle=-\left\langle K^\sharp,I_\alpha X_a\right\rangle\left\langle K^\sharp,I_\beta X_b\right\rangle\left\langle X_b, I_\alpha I_\beta X_a\right\rangle\\=&\left\langle K^\sharp,I_\alpha X_a\right\rangle\left\langle K^\sharp,I_\alpha X_b\right\rangle\left\langle X_b,X_a\right\rangle-\sum_{\alpha\neq\beta}\left\langle K^\sharp,I_\alpha X_a\right\rangle\left\langle K^\sharp,I_\beta X_b\right\rangle\left\langle X_b, I_\alpha I_\beta X_a\right\rangle\\
=&7|K|^2-42|K|^2=-35|K|^2.
\end{aligned}\end{equation*}}
By Proposition \ref{H}, Proposition \ref{V} and $-[X_a,X_b]=T_{X_a,X_b}=\sum_\alpha\hbox{d}\theta_\alpha(X_a,X_b)R_\alpha,$ we have
\begin{equation}\begin{aligned}\label{3.33}
-&\left\langle A_{[X_a,X_b]}X_b,
X_a\right\rangle=\hbox{d}\theta_\alpha(X_a,X_b)\left\langle A_{R_\alpha}X_b,
X_a\right\rangle\\=&\hbox{d}\theta_\alpha(X_a,X_b)
\left\langle\left(K(R_\alpha)+(U_{R_\alpha})_{\mathfrak{so}(7)}\right)X_b,
X_a\right\rangle\\=&\left\langle I_\alpha X_a,X_b\right\rangle\left\{\left\langle K(R_\alpha)X_b,X_a\right\rangle+
2\left\langle \left(I_\alpha\nabla K^\sharp\right)_{\mathfrak{so}(7)}(X_b),X_a\right\rangle
\right.\\&\left.+4\sum_{\beta}\left\langle \left[(I_\beta K) I_\alpha I_\beta K^\sharp\right]_{\mathfrak{so}(7)}(X_b),X_a \right\rangle -4\sum_{\beta\neq \alpha}\left\langle \left[(I_\beta I_\alpha K) I_\beta K^\sharp\right]_{\mathfrak{so}(7)}(X_b),X_a \right\rangle \right\}.
\end{aligned}\end{equation}
We claim (\ref{3.33}) vanishes. Set $\left(I_\beta K\right)(X_b):=\left\langle I_\beta K^\sharp,X_b\right\rangle.$ Then substitute the above identities into (\ref{corf}) to get (\ref{sgg}). Note that the projection of $I_\alpha\nabla K^\sharp$ to $\mathbb R^7$ is $\left(I_\alpha\nabla K^\sharp\right)_\beta I_\beta,$ then
we have
$$\left\langle I_\alpha X_a,X_b\right\rangle\left\langle \left(I_\alpha\nabla K^\sharp\right)_{\mathbb{R}^7}(X_b),X_a\right\rangle=\frac18\left\langle I_\alpha X_a,X_b\right\rangle\left\langle I_\alpha\nabla_{X_c}K^\sharp,I_\gamma X_c\right\rangle\left\langle I_\gamma X_b,X_a\right\rangle=-7{\rm tr}^H\nabla K,$$ and so \begin{equation*}\begin{aligned}
&\left\langle I_\alpha X_a,X_b\right\rangle\left\langle \left(I_\alpha\nabla K^\sharp\right)_{\mathfrak{so}(7)}(X_b),X_a\right\rangle\\=&
\left\langle I_\alpha X_a,X_b\right\rangle\left\{\left\langle \left(I_\alpha\nabla K^\sharp\right)(X_b),X_a\right\rangle-\left\langle \left(I_\alpha\nabla K^\sharp\right)_{\mathbb R^7}(X_b),X_a\right\rangle\right\}=0.
\end{aligned}\end{equation*}
 We have
{\begin{equation*}\begin{aligned}
&\left\langle I_\alpha X_a,X_b\right\rangle\left\{\sum_{\beta}\left\langle I_\beta K^\sharp,X_b\right\rangle \left\langle I_\alpha I_\beta K^\sharp,X_a\right\rangle-\sum_{\beta\neq \alpha}\left\langle I_\beta I_\alpha K^\sharp,X_b\right\rangle\left\langle I_\beta K^\sharp,X_a\right\rangle\right\}\\&=-\left\langle I_\alpha X_a,X_b\right\rangle\left\langle I_\alpha K^\sharp,X_b\right\rangle \left\langle K^\sharp,X_a\right\rangle=-7|K|^2,
\end{aligned}\end{equation*}}
and
{\begin{equation*}\begin{aligned}&\left\langle I_\alpha X_a,X_b\right\rangle\left\{\sum_\beta\left\langle \left[(I_\beta K) I_\alpha I_\beta K^\sharp\right]_{\mathbb{R}^7}(X_b),X_a \right\rangle-\sum_{\beta\neq \alpha}\left\langle \left[(I_\beta I_\alpha K) I_\beta K^\sharp\right]_{\mathbb{R}^7}(X_b),X_a \right\rangle\right\}\\=&\frac18\left\langle I_\alpha X_a,X_b\right\rangle\langle I_\beta K^\sharp,X_n\rangle\langle I_\alpha I_\beta K^\sharp,I_\gamma X_n\rangle\langle I_\gamma X_b,X_a\rangle\\&-\frac18\sum_{\beta\neq \alpha}\left\langle I_\alpha X_a,X_b\right\rangle\langle I_\beta I_\alpha K^\sharp,X_n\rangle\langle I_\beta K^\sharp,I_\gamma X_n\rangle\langle I_\gamma X_b,X_a\rangle=-7|K|^2.
\end{aligned}\end{equation*}}
The right hand side does not vanish only when $\alpha=\gamma.$
So \begin{align*}\left\langle I_\alpha X_a,X_b\right\rangle\left\{\sum_{\beta}\left\langle \left[(I_\beta K) I_\alpha I_\beta K^\sharp\right]_{\mathfrak{so}(7)}(X_b),X_a \right\rangle -\sum_{\beta\neq \alpha}\left\langle \left[(I_\beta I_\alpha K) I_\beta K^\sharp\right]_{\mathfrak{so}(7)}(X_b),X_a \right\rangle\right\}=0.\end{align*}
Thus (\ref{3.33}) vanishes. The Proposition is proved.
\qed
\subsection{The transformation formula for the  OC Yamabe operator}
Since $\nabla$ preserves $H,$ there exist $1$-forms $\omega_{a}^{b},$ such that
$\nabla_{Y}V_{a}=\omega_{a}^{b}(Y)V_{b},$ for $Y\in H,$ where $\{V_a\}$ is a local basis. Write $\omega_{a}^{b'}(V_{b})=\Gamma_{ba}^{\ \ b'}.$ Then $\nabla_{V_{a}}V_{b}
=\Gamma_{ab}^{\ \ c}V_{c}.$ For $1$-form $\omega\in\Omega(M),$
$(\nabla_{X}\omega)(Y)={X}(\omega(Y))-\omega(\nabla_{X}Y).$ The Carnot-Carath\'eodory metric $g$ induces a dual metric on $H^{*}$, denoted  by $\langle\cdot,\cdot\rangle_{g}$. Then  we   define an $L^{2}$ inner product $\langle \cdot,\cdot\rangle_{g}$ on $\Gamma(H^{*})$ by
\begin{align*}
\langle \omega,\omega'\rangle_{g}:=\int_{M}\langle \omega,\omega'\rangle\hbox{d}V_{g},
\end{align*}
where the volume form $\hbox{d}V_{g}$  is
\begin{align}\label{320}
\hbox{d}V_{g}:={\theta}_{1}\wedge\cdots\wedge
{\theta}_{7}\wedge(\hbox{d}{\theta}_{\beta})^{4},\end{align}
$\beta=1,\cdots,7,$ if we write $\Theta=({\theta}_{1},\cdots,{\theta}_{7})$ locally.

\begin{prop}\label{234}
The volume element ${\rm{d}}V_{g}$  only depends on $g$, not on $\beta$ or the choice of the $\mathbb{R}^{7}$-valued contact form
$\Theta=(\theta_{1},\cdots,\theta_{7}).$
\end{prop}
\begin{proof}
Let  $1$-forms $\{\theta^a\}$ be  the basis dual to $\{V_{a}\}.$  Since
 \begin{align*}
 \hbox{d}\theta_\beta\left(V_a,V_b\right)=g\left(I_\beta V_a,V_b\right)=g\left(E^\beta_{ca}V_c,V_b\right) =E^\beta_{ba},
 \end{align*} we have  the structure equation
\begin{align*}
\hbox{d}\theta_{\beta}=\frac12 E_{ba}^{\beta}\theta^a
\wedge\theta^b,\quad {\rm mod}\  \theta_{1},\cdots,\theta_{7},  \end{align*}
$\beta=1,\cdots,7,$ where $E^{\beta}$ is  given in (\ref{o2.12}).
It is direct to check that
\begin{align}\label{3.7}
(\hbox{d}{\theta}_{\beta})^{4}&=\left(\frac12E_{ba}^{\beta}\theta^a
\wedge\theta^b\right)^{4}
=4!\theta^0\wedge\cdots\wedge\theta^7,\quad {\rm mod}\  \theta_{1},\cdots,\theta_{7}.
\end{align}
 Let $\tilde{\Theta}=(\tilde{\theta}_{1},\cdots,\tilde{\theta}_{7})$ be another contact form satisfying $\hbox{d}\tilde{\theta}_{\beta}(X,Y)=g(I_{\beta}X,Y).$ We can write $\tilde{\theta}_{\beta}=c_{\beta\alpha}\theta_{\alpha},$ for some ${\rm SO}(7)$-valued function $(c_{\alpha\beta})$ and simultaneously $\tilde{I}_{\beta}=c_{\beta\alpha}I_{\alpha},$ for $\beta=1,\cdots,7.$ Dually, we have $\tilde{V}_{a}=k_{ab}V_{b},$ for some induced ${\rm SO}(8)$-valued function $(k_{ab}),$ and the dual basis $\{\tilde{\theta}^{a}\}$ such that
$
\tilde{\theta}^{a}=k_{ab}\theta^{b}
$
for $a=0,\cdots,7.$ In fact, we have $\tilde{\theta}^{0}\wedge\cdots\wedge\tilde{\theta}^{7}=\det(k_{ab})
{\theta}^{0}\wedge\cdots\wedge{\theta}^{7}$ and $\det(k_{ab})=1$ by $(k_{ab})\in {\rm SO}(8).$ By (\ref{3.7}), we have
\begin{equation}\label{3.77}
\begin{aligned}
\hbox{d}V_{\tilde{g}}&=\tilde{\theta}_{1}\wedge\cdots
\wedge\tilde{\theta}_{7}\wedge\left({\rm d}\tilde{\theta}_{\beta}
\right)^{4}=4!\det(c_{\alpha\beta})\theta_{1}\wedge\cdots\wedge
\theta_{7}\wedge\tilde{\theta}^{0}\wedge\cdots\wedge\tilde{\theta}^{7}   \\&=4!{\theta}_{1}\cdots\wedge
{\theta}_{7}\wedge\theta^{0}\wedge\cdots\wedge\theta^{7}=
{\theta}_{1}\cdots\wedge
{\theta}_{7}\wedge({\rm d}{\theta}_{\beta})^{4}.
\end{aligned}
\end{equation}
The proposition is proved.
\end{proof}
Denote $\hbox{d}_{b}:={\rm pr}\circ \hbox{d},$ where ${\rm pr}$ is the projection from $T^{*}M$ to $H^{*}$. We define the \emph{SubLaplacian} $\Delta_{g}$ associated to Carnot-Carath\'eodary metric $g$  by
\begin{align}\label{3.9}
\int_{M}\Delta_{g}u\cdot v\hbox{d}V_{g}=
\int_{M}\langle \hbox{d}_{b}u,\hbox{d}_{b}v\rangle\hbox{d}V_{g}
\end{align}
for $u,v\in C_{0}^{\infty}(M),$ where $\hbox{d}V_{g}$ is the volume form. %Since $\Gamma_{ab}^{\ \ c}=-\Gamma_{ac}^{\ \ b},$
The SubLaplacian $\Delta_{g}$ has the following expression (see Proposition 2.1 in \cite{wang2} for the qc case.)
\begin{prop}
Let $\{V_{a}\}_{a=0}^{7}$ be a local  orthonormal basis of $H.$ Then locally for $u\in C^{\infty}(M),$ we have
\begin{align}\label{deltag}
\Delta_{g}u=\left(-V_{a}V_{a}u+\Gamma_{bb}^{\ \ a}V_{a}u\right).
\end{align}
\end{prop}
\begin{proof}
Let $\{\theta^{a}\}_{a=0}^{7}$ be the dual basis of $\{V_{a}\}_{a=0}^{7}$ for $H^{*}$ and $\theta^{a}|_{V}=0$ for each $a.$ Let $V_{a}^{*}$ be the formal adjoint operator of $V_{a}.$ By Stokes' formula
\begin{equation}\begin{aligned}\label{i}
\int\limits_{M}V_{a}u\cdot v\hbox{d}V_{g}=&\int\limits_{M}i_{V_{a}}\hbox{d}u\cdot v\hbox{d}V_{g}=\int\limits_{M}v\hbox{d}u\wedge i_{V_{a}}\hbox{d}V_{g}\\=&-\int\limits_{M}u\hbox{d}v\wedge i_{V_{a}}\hbox{d}V_{g}-\int\limits_{M}uv\hbox{d}( i_{V_{a}}\hbox{d}V_{g})=-\int\limits_{M}u{V_{a}}v \hbox{d}V_{g}-\int\limits_{M}uv\hbox{d}( i_{V_{a}}\hbox{d}V_{g}).
\end{aligned}\end{equation}
Here we have used identities $\int_{M}i_{V_{a}}(v\hbox{d}u\wedge\hbox{d}V_{g})=0.$
Note that $$\hbox{d}\theta^{a}=\frac12(-\Gamma_{bb'}^{\ \ \ a}+
\Gamma_{b'b}^{\ \ \ a})\theta^{b}\wedge\theta^{b'},\ {\rm mod}\ \theta_{1},\cdots,\theta_{7}.$$
Note that
\begin{equation}\begin{aligned}\label{w2.39}
\hbox{d}(i_{V_{a}}\hbox{d}V_{g})=&(-1)^{a}4!
(-1)^{\alpha+1}\theta_{1}\wedge\cdots\wedge\hbox{d}\theta_{\alpha}
\wedge\cdots\wedge\theta_{7}\wedge\theta^{0}\cdots\wedge \widehat\theta^a\wedge\cdots\wedge\theta^{7}\\
&+\sum_{b\neq a}\hbox{d}\theta^{b}\wedge i_{V_{b}}i_{V_{a}}\hbox{d}V_{g}.
\end{aligned}\end{equation}
The first term on the right side of (\ref{w2.39}) is zero since it is annihilated by the Reeb vectors $R_{\alpha}$ by $i_{R_{\alpha}}\hbox{d}\theta_{\alpha}=0.$ Since  $$\hbox{d}\theta^{b}=\frac{1}{2}(-\Gamma_{ab}^{\ \ b}+
\Gamma_{ba}^{\ \ b})\theta^{a}\wedge\theta^{b}+\frac{1}{2}(-\Gamma_{ba}^{\ \ b}+
\Gamma_{ab}^{\ \ b})\theta^{b}\wedge\theta^{a}+\cdots,$$
 we get
\begin{align}\label{ww}
\hbox{d}(i_{V_{a}}\hbox{d}V_{g})=\sum_{b\neq a}(-\Gamma_{ab}^{\ \ b}+
\Gamma_{ba}^{\ \ b})\hbox{d}V_{g}=-\sum_{b\neq a}\Gamma_{bb}^{\ \ a}\hbox{d}V_{g}.
\end{align}
So, by (\ref{i}) and  (\ref{ww}), we have $V_{a}^{*}=-V_{a}+\Gamma_{bb}^{\ \ a}.$
Therefore, (\ref{deltag}) follows.
\end{proof}

The scalar curvature $s_{\tilde{g}}$ for the metric
$\tilde{g}=e^{2h}g_0$ satisfies
\begin{align}\label{2.32}
s_{\tilde{g}}=e^{-2h}\left(42\Delta_{g_0}h
-420\sum_{a=0}^7(X_ah)^{2}\right).
\end{align}
%\end{cor}
We can also write the transformation law in the following form.
\begin{cor}\label{s3.1}
The scalar curvature $s_{\tilde{g}}$ of the Biquard connection  for $\tilde{g}=\phi^{\frac{4}{Q-2}}g_0$ on the octonionic Heisenberg group satisfies the OC Yamabe equation:
\begin{align*}
b\Delta_{g_0}\phi=
s_{\tilde{g}}\phi^{\frac{Q+2}{Q-2}},\quad b=\frac{4(Q-1)}{Q-2}=\frac{21}{5}.
\end{align*}
\end{cor}

 Now let us derive the transformation formula of the OC Yamabe operator on general spherical OC manifolds. See \cite{Orsted} for such derivation for the pseudo-Riemannian case, \cite{Wang3} for the CR case and \cite{Shi} for the qc case.
\begin{prop}
Let $({M},\tilde{g},\tilde{\mathbb{I}})$ and $(M,g,\mathbb{I})$ be two  OC manifolds. Let $\tilde{g}=\phi^{\frac{4}{Q-2}}g$ for some positive smooth function $\phi$ on ${M}$. Then
\begin{align}\label{123}
  \Delta_{g}(\phi\cdot f)=\Delta_{g}\phi\cdot f+{\phi^{\frac{Q+2}{Q-2}}}\Delta_{\tilde{g}}f
\end{align}
for any smooth function $f$ on $M$.
\end{prop}

\begin{proof} Note that $\Delta_g$ and the OC Yamabe operator is independent of the choice of compatible $\tilde{\mathbb I}.$
Let $\Theta=(\theta_{1},\cdots,\theta_{7})$ be a $\mathbb{R}^{7}$-valued $1$-form associated to $({M},{g},\mathbb{I})$ and let $\tilde{\Theta}=(\tilde{\theta}_{1},\cdots,\tilde{\theta}_{7})$ be associated to $({M},\tilde{{g}},\mathbb{I})$.
For any real function $h$ on $M$, we have
\begin{equation}\label{11}
\begin{aligned}
\langle\Delta_{g}(\phi\cdot f),h\rangle_{g}
&=\langle \hbox{d}_{b}\phi\cdot f+\phi \hbox{d}_{b} f,\hbox{d}_{b}h\rangle_{g}
=\langle \hbox{d}_{b}\phi,f\cdot \hbox{d}_{b}h\rangle_{g}+
\langle \hbox{d}_{b} f,\phi\cdot \hbox{d}_{b}h\rangle_{g}\\
&=\langle \hbox{d}_{b}\phi,\hbox{d}_{b}(f\cdot h)\rangle_{g}
+\langle \hbox{d}_{b}f,\phi\cdot \hbox{d}_{b}h- h\cdot \hbox{d}_{b}\phi\rangle_{g}\\
&=\langle \Delta_{g}\phi,f\cdot h\rangle_{g}
+\langle \hbox{d}_{b}f,\phi\cdot \hbox{d}_{b}h- h\cdot \hbox{d}_{b}\phi\rangle_{g}.
\end{aligned}
\end{equation}
Let us calculate the second term in the right side of (\ref{11}). By our assumption and Proposition \ref{234}, we just need to consider $\tilde{\Theta}=\phi^{\frac{4}{Q-2}}\Theta.$ Then for a fixed $\beta$, we have
\begin{align}\label{12}
\hbox{d}\tilde{\theta}_\beta=\hbox{d}(\phi^{\frac{4}{Q-2}}\theta_\beta)
={\frac{4}{Q-2}}\phi^{\frac{6-Q}{Q-2}}\hbox{d}\phi\wedge\theta_\beta
+\phi^{\frac{4}{Q-2}}\hbox{d}\theta_\beta.
\end{align}
So, we get
$
\hbox{d}V_{\tilde{g}}
=\phi^{\frac{2Q}{Q-2}}\hbox{d}V_{g}.
$
Consequently, for $1$-forms $\omega_{1},\omega_{2}\in H^{*},$ we have
\begin{equation*}
\begin{aligned}
\langle \omega_{1},\omega_{2}\rangle_{g}=
\int_{M} \langle \omega_{1},\omega_{2}\rangle\hbox{d}V_{g}
=\langle\phi^{-2}\omega_{1},\omega_{2}\rangle_{\tilde{g}}.
\end{aligned}
\end{equation*}
Now we find that
\begin{equation*}
\begin{aligned}
\langle \hbox{d}_{b}f,\phi \cdot \hbox{d}_{b} h- h\cdot \hbox{d}_{b}\phi\rangle_{g}
&=\langle\phi^{-2} \hbox{d}_{b}f,\phi\cdot \hbox{d}_{b} h- h\cdot \hbox{d}_{b}\phi\rangle_{\tilde{g}}
=\langle \hbox{d}_{b}f,\hbox{d}_{b}(\phi^{-1}h)\rangle_{\tilde{g}}\\
&=\int \Delta_{\tilde{g}}f\cdot\phi^{-1}h\hbox{d} V_{\tilde{g}}
=\int \phi^{\frac{Q+2}{Q-2}}\Delta_{\tilde{g}}f\cdot h\hbox{d} V_{g}.
\end{aligned}
\end{equation*}
The proposition is proved.
\end{proof}
{\it Proof of Theorem \ref{s-curva}.}
By choosing an open set of the octonionic Heisenberg group with standard OC metric  as a local coordinate,  we can write $\tilde g=\phi_1^{\frac{4}{Q-2}}g_0,$ $g=\phi_2^{\frac{4}{Q-2}}g_0$ locally. Then $\tilde g=\phi^{\frac{4}{Q-2}}g$ with $\phi=\phi_1\phi_2^{-1}.$ Applying (\ref{123})
 we have
\begin{align*}
\Delta_{g_0}\phi_1=\Delta_{g_0}(\phi_2\cdot\phi_1\phi_2^{-1})=\Delta_{g_0}\phi_2\cdot \phi_1\phi_2^{-1}+\phi_2^{\frac{Q+2}{Q-2}}\Delta_{g}(\phi_1\phi_2^{-1}).
\end{align*}
Thus we have
\begin{align*}
{s_{\tilde g}}\phi_1^{\frac{Q+2}{Q-2}}={s_{g}}\phi_2^{\frac{Q+2}{Q-2}}\cdot\phi+
{b}\phi_2^{\frac{Q+2}{Q-2}}\Delta_g\phi
\end{align*}
by using Corollary \ref{s3.1}, i.e. $b\Delta_{g}\phi+s_g\phi=
s_{\tilde{g}}\phi^{\frac{Q+2}{Q-2}}$ . The theorem is proved.
\qed
\vskip 5mm
{\it Proof of Corollary \ref{L-trans}.}
By using (\ref{122}) and (\ref{123}), we have
\begin{equation*}
\begin{aligned}
L_{g}(\phi  f)&=b\Delta_{g}(\phi f)+s_{g}\phi f=b\left(\Delta_{g}\phi\cdot f+\phi^{\frac{Q+2}{Q-2}}\Delta_{\tilde{g}}f\right)+s_{g}\phi f=\phi^{\frac{Q+2}{Q-2}}\left(b\Delta_{\tilde{g}}f+s_{\tilde{g}}f\right).
\end{aligned}
\end{equation*}
The result follows.
\qed
\vskip 3mm
 Define the \emph{OC Yamabe invariant}
\begin{align}\label{4.7}
\lambda(M,g):=\inf_{u>0}
\frac{\int_{M}\left(b|\nabla_{g}u|^{2}+s_{g}
u^{2}\right)\hbox{d}V_{g}}{\left(\int_{M}u^{\frac{2Q}{Q-2}}
\hbox{d}V_{g}\right)^{\frac{Q-2}{Q}}},
\end{align}
where $|\nabla_{g}f|^{2}=\sum_{a=0}^{7}|V_af|^{2}$ if $\{V_a\}$ is a local orthogonal basis of $H$ under the CC metric $g.$ It is
an invariant  for the conformal  class of OC manifolds. There is a natural OC Yamabe problem as in the (contact) Riemannian, CR and qc  cases
(cf. e.g. \cite{IMV, Jerison, ww} and references therein). The Yamabe-type  equation on groups of Heisenberg type, including the octonionic Heisenberg group, has been studied in  \cite{GV}.

\section{The Green function of the OC Yamabe operator and conformal invariants}
\subsection{The Green function}
It is a  continuous function $G_{g}:M\times M\backslash {\rm diag}M\rightarrow \mathbb{R}$ such that
\begin{align*}
\int_{M}G_{g}(\xi,\eta)L_{g}u(\eta)\hbox{d}V_{g}(\eta)=u(\xi)
\end{align*}
for all $u\in C_{0}^{\infty}(M).$ Namely, $L_{g}G_{g}(\xi,\cdot)=\delta_{\xi}.$

The explicit form of the fundamental solution of the SubLaplacian   on H-type groups, including  the  octonionic Heisenberg groups, is known. (cf. e.g. \cite{Bonfiglioli, kaplan}). It can be checked directly as in the Appendix in \cite{Shi}.
\begin{prop}\label{1.10}
The Green function of the OC Yamabe operator $L_{0}=b\Delta_{0}$ on the octonionic Heisenberg group $\mathscr{H}$ with the pole at $\xi$ is
\begin{align*}
G_{0}(\xi,\eta):=\frac{C_{Q}}{\|\xi^{-1}\eta\|^{Q-2}},
\end{align*}
for $\xi\neq \eta,\ \xi,\eta\in \mathscr{H},$ where $\|\cdot\|$ is the norm on $\mathscr{H}$ defined by (\ref{124}) and
\begin{align}\label{146}
{C_{Q}}^{-1}=(Q+2)(Q-2)b
\int_{\mathbb{R}^{15}}\frac{|x|^{2}}{(|x|^{4}+|t|^{2}+1)^{7}}{\rm d}V_{0},
\end{align}
where ${\rm d}V_{0}$ is {Lebesgue measure}.
\end{prop}

\begin{prop}\label{2.1}
For a connected compact OC manifold $(M,g,\mathbb{I}),$ we have trichotomy: there exists an OC metric $\tilde{g}$ conformal to $g$  which has either positive, negative or  vanishing scalar curvature  everywhere.
\end{prop}

\begin{proof}
The OC Yamabe operator $L_{g}$ is a formally self-adjoint and subelliptic differential operator. So its spectrum is real and bounded from below. Let $\lambda_{1}$ be the first eigenvalue of $L_{g}$ and let $\phi$ be an eigenfunction of $L_{g}$ with eigenvalue $\lambda_{1}$. Then $\phi>0$ and is  $C^{\infty}$ as the qc case \cite{Shi}. The scalar curvature of $(M,\tilde{g},\mathbb{I})$ with $\tilde{g}=\phi^{\frac{4}{Q-2}}g$ is $s_{\tilde{g}}=\lambda_{1}\phi^{-\frac{4}{Q-2}}$ by the OC Yamabe equation (\ref{122}). In particular, $s_{\tilde{g}}>0$ (resp. $s_{\tilde{g}}<0,$  resp. $s_{\tilde{g}}\equiv0$) if $\lambda_{1}>0$  (resp. $\lambda_{1}<0$, resp. $\lambda_{1}=0$). On the other hand, if $\hat{g}$ has scalar curvature $s_{\hat{g}}>0$ (resp. $s_{\hat{g}}<0,$ resp. $s_{\hat{g}}\equiv0$),
the first eigenvalue ${\hat{\lambda}}_{1}$ of $L_{\hat{g}}$ obviously satisfies ${\hat{\lambda}}_{1}>0$  (resp. ${\hat{\lambda}}_{1}<0$, resp.
${\hat{\lambda}}_{1}=0$).
\end{proof}
For $\xi\in\mathscr{H}$ and $\epsilon>0,$ define a {\it ball $B(\xi,\epsilon):=\{\eta\in\mathscr{H};\|\xi^{-1}\cdot\eta\|<\epsilon\}$ in the octonionic Heisenberg group.}

\begin{prop}\label{p3.7}
Let $(M,g,\mathbb{I})$ be a connected  compact OC manifold with positive scalar curvature and let $U$ be a sufficiently small open set. Then the function $G_{g}(\xi,\eta)-\rho_{g}(\xi,\eta)$ can be extended to a $C^{\infty}$ function on $U\times U$, where $\rho_{g}(\cdot,\cdot)$ is
 given by (\ref{222}).
\end{prop}

\begin{proof}
Suppose that $\bar{U}\subset \tilde{U}\subset \mathscr{H}$ and $g=\phi^{\frac{4}{Q-2}}g_{0}$ on $\tilde{U}$. We choose a sufficiently small $\rho$ such that $B(\xi,\rho)\subset\tilde{U}$ for any $\xi\in U$. We can construct the Green function as follows. For $\xi, \eta\in U$, define
\begin{align*}
\tilde{G}(\xi, \eta)=\tilde{G}(\xi^{-1}\eta),
\end{align*}
where $\tilde{G}$ is the cut-off fundamental solution of $L_{0}=-{b}\sum_{a=0}^{7}X_a^2,$ i.e.  $\tilde{G}(\tilde{\eta})=\frac{C_{Q}}{\|\tilde{\eta}\|^{Q-2}}
f(\tilde{\eta})$  for $\tilde{\eta}\in \mathscr{H}$ with $f\in C_{0}^{\infty}(\mathscr{H})$  satisfying $f\equiv 1 $ on $B(0,\frac{\rho}{2})$ and $f\equiv 0$ on $B(0,{\rho})^{c}$.
Then,
\begin{equation}\label{3.30}
\begin{aligned}
L_{0}\tilde{G}(\tilde{\eta})=\delta_{0}-{b}X_a
\left(\frac{C_{Q}}{\|\tilde{\eta}\|^{Q-2}}\right)X_af(\tilde{\eta})+\frac{C_{Q}}
{\|\tilde{\eta}\|^{Q-2}}L_{0}f(\tilde{\eta})=:\delta_{0}+\tilde{G}_{1}(\tilde{\eta})
\end{aligned}
\end{equation}
by $X_af\equiv0$ on $B(0,\frac{\rho}{2})$. Here $\delta_{0}$ is the Dirac function at the origin with respect to  the measure $\hbox{d}V_{0}$ and $\tilde{G}_{1}$ is defined by the last equality in (\ref{3.30}).
Set ${G}_{1}(\xi,\eta):=\tilde{G}_{1}(\xi^{-1}\eta)$ for $\xi,\eta\in U$. Then, ${G}_{1}(\xi,\eta)\in C_0^{\infty}(\widetilde{U}\times \widetilde{U})$ and for each $\xi\in U$, ${G}_{1}(\xi,\cdot)$ can be naturally extended to a smooth function on $M$. By
transformation law (\ref{o121}) and left invariance of $X_a$, we find that
\begin{equation*}
\begin{aligned}
L_{g}\left(\phi(\xi)^{-1}\phi(\cdot)^{-1}{\tilde{G}}(\xi,\cdot)\right)
&=\phi(\xi)^{-1}\phi(\cdot)^{-\frac{Q+2}{Q-2}}L_{0}{\tilde{G}}(\xi,\cdot)
=\phi(\xi)^{-1}\phi(\cdot)^{-\frac{Q+2}{Q-2}}\left(\delta_{0}(\xi^{-1}\cdot)+
{G}_{1}(\xi,\cdot)\right)\\ &=\delta_{\xi}+\phi(\xi)^{-1}\phi(\cdot)^{-\frac{Q+2}{Q-2}}{G}_{1}(\xi,\cdot),
\end{aligned}
\end{equation*}
on $U$ for $\xi\in U$ , where $\delta_{\xi}$ is the Dirac function at point $\xi$  with respect to  the measure $\hbox{d}V_{g}=
\phi^{\frac{2Q}{Q-2}}\hbox{d}V_{0}$. Now  set
\begin{align}\label{311}
G(\xi,\eta):=\phi(\xi)^{-1}\phi(\eta)^{-1}{\tilde{G}}(\xi,\eta)+{G}_{2}
(\xi,\eta)
\end{align}
for ${\eta}\in M$, where ${G}_{2}(\xi,\eta)$ satisfies
\begin{align}\label{3.32}
L_{g}{G}_{2}(\xi,\cdot)
=-\phi(\xi)^{-1}\phi(\cdot)^{-\frac{Q+2}{Q-2}}{G}_{1}(\xi,\cdot).
\end{align}
$G_{2}(\xi,\cdot)$ exists since $L_{g}$ is invertible in $L^{2}(M)$. $G_{2}(\xi,\cdot)\in C^{\infty}(M)$ for fixed $\xi\in U$ by the subelliptic  regularity of $L_{g}$. $G_{2}(\cdot,\eta)$ is also in $C^{\infty}(U)$ by differentiating (\ref{3.32}) with respect to the variable $\xi$ repeatedly. Now we have $L_{g}G(\xi,\cdot)=\delta_{\xi}$, i.e.  $G(\xi,\eta)$ is the Green function  of $L_{g}$. By  (\ref{311}),  $G_{g}(\xi,\eta)-\rho_{g}(\xi,\eta)\in C^{\infty}(U\times U).$
\end{proof}

We have the  following transformation formula of the Green  functions under conformal OC transformations.
\begin{prop}\label{p3.8}
Let $(M,g,\mathbb{I})$ be a connected, compact, scalar positive OC manifold  and $G_{g}$ be the Green function of the OC Yamabe operator $L_{g}.$ Then
\begin{align}\label{32}
G_{\tilde{g}}(\xi,\eta)=\frac{1}{\phi(\xi)\phi(\eta)}G_{g}(\xi,\eta)
\end{align}
is the Green function of the OC Yamabe operator $L_{\tilde{g}}$ for $\tilde{g}=\phi^{\frac{4}{Q-2}}g.$
\end{prop}

\begin{proof}
By   the transformation law (\ref{o121}), we find that
\begin{equation*}
\begin{aligned}
\int_{M}\frac{G_{g}(\xi,\eta)
L_{\tilde{g}}u(\eta)}{\phi(\xi)\phi(\eta)} \hbox{d}V_{\tilde{g}}&=\frac{1}{\phi(\xi)}\int_{M}\frac{1}{\phi(\eta)}
G_{g}(\xi,\eta)\phi(\eta)^{-\frac{Q+2}{Q-2}}L_{g}(\phi u)(\eta)
\phi(\eta)^{\frac{2Q}{Q-2}}\hbox{d}V_{g}\\
&=\frac{1}{\phi(\xi)}\int_{M}G_{g}(\xi,\eta)L_{g}(\phi u)(\eta)\hbox{d}V_{g}=u(\xi)
\end{aligned}
\end{equation*}
for any $u\in C_{0}^{\infty}(M)$. The proposition follows form the uniqueness of the Green function.
\end{proof}

\subsection{An invariant tensor on a scalar positive  OC manifold}
{\it Proof of Theorem \ref{can-g}.}
We will verify that $\mathcal{A}_{g}$ is independent of the choice of local coordinates and $\mathcal{A}_{g}^{2}g$ is independent of the choice of $g$ in the conformal class $[g]$. Suppose $\tilde{g}=\Phi^{\frac{4}{Q-2}}g$. Let $U\subset M$ be an open set and let $\rho:U\rightarrow V\subset\mathscr{H}$ and $\tilde{\rho}:U\rightarrow \tilde{V}\subset\mathscr{H}$ be two coordinate charts such that
\begin{equation*}
\begin{aligned}
g=\rho^{*}\left(\phi_{1}^{\frac{4}{Q-2}}g_{0}\right),\quad
\tilde{g}=\tilde{\rho}^{*}\left(\phi_{2}^{\frac{4}{Q-2}}g_{0}\right),
\end{aligned}
\end{equation*}
for two positive function $\phi_{1}$ and $\phi_{2}$. Then, $f=\tilde{\rho}\circ\rho^{-1}: V\rightarrow \tilde{V}$  and
\begin{align}\label{fg0}
\left.f^{*}g_{0}\right|_{\xi'}=\phi^{\frac{4}{Q-2}}(\xi')
\left.g_{0}\right|_{\xi'}\quad {\rm with} \ \phi(\xi')=\phi_{1}(\xi')\phi_{2}^{-1}(f(\xi'))\Phi(\rho^{-1}(\xi')),
\end{align}
for $\xi'\in V$. We claim the following the transformation law of the Green function on the octonionic  Heisenberg group under a conformal OC transformation:
\begin{align}\label{14}
\frac{1}{\|f(\xi')^{-1}f(\eta')\|^{Q-2}}=\frac{1}{\phi(\xi')\phi(\eta')}
\cdot\frac{1}{\|\xi'^{-1}\eta'\|^{Q-2}},
\end{align}
for any $\xi', \eta'\in V.$
Apply this to $\xi'=\rho(\xi), \eta'=\rho(\eta)$ and $f=\tilde{\rho}\circ\rho^{-1}$ to get
\begin{align*}
\frac{1}{\|\tilde{\rho}(\xi)^{-1}\tilde{\rho}(\eta)\|^{Q-2}}=
\frac{1}{\phi(\rho(\xi))}\frac{1}{\phi(\rho(\eta))}
\frac{1}{\|\rho(\xi)^{-1}\rho(\eta)\|^{Q-2}}
\end{align*}
and so
\begin{equation*}
\begin{aligned}
\mathcal{A}_{\tilde{g}}(\xi)&=\lim_{\eta\rightarrow \xi}\left|G_{\tilde{g}}(\xi,\eta)-\frac{1}{
\phi_{2}(\tilde{\rho}(\xi))\phi_{2}(\tilde{\rho}(\eta))}\cdot\frac{C_{Q}}{\|
\tilde{\rho}(\xi)^{-1}\tilde{\rho}(\eta)\|^{Q-2}}\right|^{\frac{1}{Q-2}}\\
&=\lim_{\eta\rightarrow\xi}\left|\frac{G_{g}(\xi,\eta)}{\Phi(\xi)\Phi(\eta)}-\frac{1}{
\Phi(\xi)\Phi(\eta)\phi_{1}(\rho(\xi))\phi_{1}(\rho(\eta))}\cdot\frac{C_{Q}}{\|\rho(\xi)^{-1}
\rho(\eta)\|^{Q-2}}\right|^{\frac{1}{Q-2}}\\
&=\Phi^{-\frac{2}{Q-2}}(\xi)\lim_{\eta\rightarrow\xi}\left|G_{g}
(\xi,\eta)-\frac{1}{\phi_{1}({\rho}(\xi))\phi_{1}
(\rho(\eta))}\cdot\frac{C_{Q}}{\|\rho(\xi)^{-1}\rho(\eta)\|^{Q-2}}\right|^{\frac{1}{Q-2}}=\Phi^{-\frac{2}{Q-2}}(\xi)\mathcal{A}_{g}(\xi).
\end{aligned}
\end{equation*}
Consequently, we have $\mathcal{A}_{\tilde{g}}^{2}\tilde{g}=\mathcal{A}_{g}^{2}{g}.$

It remains to check (\ref{14}).
By OC Liouville-type Theorem \ref{liou}, $f$ is a restriction to $V$ of an OC automorphism of $\mathscr{H}$, denoted  also by $f$. By the transformation law (\ref{o121}), for functions ${\tilde{\phi}}:=\phi\circ f^{-1},\tilde{u}:=u\circ f^{-1}$ on $\tilde{V},$ we have
\begin{equation}
\begin{aligned}\label{o15}
\left.L_{0}\left(\tilde{\phi}^{-1}\tilde{u}\right)\right|_{f(\eta')}&=
\left.L_{g}\left(\phi^{-1}u\right)\right|_{\eta'}=
\phi^{-\frac{Q+2}{Q-2}}(\eta')
\left.L_{0}(u)\right|_{\eta'},\\
\left.f^{*}\hbox{d}V_{0}\right|_{f(\eta')}&
=\phi^{\frac{2Q}{Q-2}}(\eta')\left.\hbox{d}V_{0}\right|_{\eta'}.
\end{aligned}
\end{equation}
 The first identity follows from (\ref{fg0}) and  the fact that the OC Yamabe operator is independent of the choice of coordinate charts. Then, by (\ref{o15}) and taking transformation $f(\eta')\rightarrow\hat{\eta}$, we find that for any $u\in C_{0}^{\infty}(\mathscr{H})$
\begin{equation*}
\begin{aligned}
\int_{\mathscr{H}}\frac{C_{Q}\phi(\xi')
\phi(\eta')}{\|f(\xi')^{-1}f(\eta')\|^{Q-2}}
L_{0}u(\eta')\hbox{d}V_{{0}}(\eta')&=\int_{\mathscr{H}}
\frac{C_{Q}\phi(\xi')}
{\|f(\xi')^{-1}f(\eta')\|^{Q-2}}
\left.L_{0}\left({\tilde{\phi}}^{-1}\tilde{u}\right)\right|_{f(\eta')}
f^{*}\hbox{d}V_{0}(\eta')\\&=\int_{\mathscr{H}}
\frac{C_{Q}\phi(\xi')}{\|f(\xi')^{-1}\hat{\eta}\|^{Q-2}}
\left.L_{0}\left({\tilde{\phi}}^{-1}\tilde{u}\right)\right|_{\hat{\eta}}
\hbox{d}V_{0}(\hat{\eta})=u(\xi').
\end{aligned}
\end{equation*}
Now by the uniqueness of the Green  function of $L_{0}$, we find that
$$G_{0}(\xi',\eta')=\frac{C_{Q}\phi(\xi')\phi(\eta')}
{\|f(\xi')^{-1}f(\eta')\|^{Q-2}}.$$ Thus, (\ref{14}) follows. The theorem is proved.
\qed
\vskip 5mm
See \cite{Leutwiler} for the identity (\ref{14}) on the Euclidean space and see \cite{wang1} on the Heisenberg group and \cite{Shi} on the quaternionic Heisenberg group. The OC positive mass  conjecture implies that $\mathcal{A}_{g}$ is non-vanishing. Then $\mathcal{A}_{g}^{2}g$ is a conformally invariant OC metric. This invariant metric was
given by Habermann-Jost \cite{habermann, habermann1}  for  the locally conformally flat case, by the second author for  the  CR case \cite{wang1} and by us \cite{Shi} for the qc case, respectively.

\subsection{The connected sum of two scalar positive OC manifolds}
Let $(M,g,\mathbb{I})$ be a  OC manifold of dimension $15$ with two punctures $\eta_{1},\eta_{2}$, or disjoint union of two connected OC manifolds $(M_{(i)},g_{(i)},\mathbb{I}_{(i)})$  with one puncture $\eta_{i}\in M_{(i)}$ each, $i=1,2.$ Let $U_{1}$ and $U_{2}$ be two disjoint neighborhoods of $\eta_{1}$ and $\eta_{2}$, respectively. Let
\begin{align}\label{265}
\psi_{i}:U_{i}\rightarrow B(0,2),\quad i=1,2,
\end{align}
be local coordinate charts such that $\psi_{i}(\eta_{i})=0$.
For $t<1$, define
\begin{equation*}
\begin{aligned}
U_{i}(t,1):=\left\{\eta\in U_{i};t<\|\psi_{i}(\eta)\|<1\right\},\quad
U_{i}(t):=\left\{\eta\in U_{i};\|\psi_{i}(\eta)\|<t\right\},
\end{aligned}
\end{equation*}
$i=1,2.$ For any $t\in(0,1), A\in {\rm Spin}(7)$, we can form a new  OC manifold $M_{t,A}$ by removing the closed balls $\overline{U_{i}(t)},i=1,2,$ and gluing  $U_{1}(t,1)$ with $U_{2}(t,1)$ by the conformal OC  mapping $\Psi_{t,A}:U_{1}(t,1)\rightarrow
U_{2}(t,1)$ defined by
\begin{align}\label{2.60}
\Psi_{t,A}(\eta)=\psi_{2}^{-1}\circ D_{t}\circ R\circ {A}\circ\psi_{1}(\eta), \ {\rm for} \ \eta\in U_{1}(t,1),
\end{align}
where $R:\{\zeta\in \mathscr{H};t<\|\zeta\|<1\}\rightarrow\{\zeta\in \mathscr{H};1<\|\zeta\|<\frac{1}{t}\}$  is the inversion  in  (\ref{44}). Note that $\Psi_{t,A}$ is a conformal OC mapping $U_{1}(t,1)$ to $U_{2}(t,1),$ which identifies the inner boundary of $U_{1}(t,1)$ with the outer boundary $U_{2}(t,1)$ and vise versa.
Let $\pi_{t,A}:(M_{1}\backslash \overline{U_{1}(t)})\cup (M_{2}\backslash \overline{U_{2}(t)})\rightarrow M_{t,A}$
be a canonical projection.  We call $M_{t,A}$ the \emph{connected sum} of $M_{1}$ and $M_{2}$. We denote this  OC manifold by $(M_{t,A},g,\mathbb{I}_{t,A})$, where $g$ is a metric in the conformal class. As in the locally conformally case, the connected sums are expected to be not isomorphic for some different choices of $t,A$ \cite{Izeki1}. Scalar positive  OC manifolds are abundant by the following proposition.
\begin{prop}\label{51}
If $t$ is sufficiently small, the connected sum $(M_{t,A},g,\mathbb{I}_{t,A})$ is scalar positive.
\end{prop}

Proposition \ref{51} follows from the following Proposition \ref{psum}. As preparation, we prove the  following lemma firstly.
\begin{lem}\label{ls}
For $g|_{\xi}=\frac{g_{0}|_{\xi}}{\|\xi\|^{2}},\xi=(x,t),$ we have
\begin{equation}\begin{aligned}\label{8899}
s_{g}(\xi)=(Q-2)(Q-1)\frac{|x|^{2}}{\|\xi\|^{2}}.
\end{aligned}\end{equation}
\end{lem}
\begin{proof}
Note that
\begin{align}\label{app1'}
X_{a}\|\xi\|^{4}=4|x|^{2}x_{a}+4E_{ba}^{\beta}x_{b}t_{\beta},
\end{align}
by using the expression of the vector field $X_{a}$ in (\ref{2.14}). Then
{\small\begin{equation}\label{A2}
\begin{aligned}
\left|\nabla_{0}\|\xi\|^{4}\right|^{2}=\left(X_{a}\|\xi\|^{4}\right) \cdot\left(X_{a}\|\xi\|^{4}\right)
=16\left(|x|^{4}x_{a}\cdot x_a+E_{b'a}^{\beta'}
E_{ba}^{\beta}x_{b'}x_{b}t_{\beta}t_{\beta'}\right)=16\|\xi\|^{4}|x|^{2},
\end{aligned}
\end{equation}}
by using (\ref{app1'}) and $E_{ba}^{\beta}E_{ab'}^{\beta'}=(E^{\beta}E^{\beta'})_{bb'},$  antisymmetry of $E^{\beta}E^{\beta'}$ of $\beta\neq \beta'$ and
$\left(E^{\beta}\right)^{2}=-{\rm id}$ in Proposition \ref{p2.4}. Similarly, by (\ref{app1'}), we get
\begin{equation}\label{A1}
\begin{aligned}
\Delta_{0}\|\xi\|^{4}=-\sum\left( {8x_{a}^{2}+4|x|^{2}+8E_{b'a}^{\beta}
E_{ba}^{\beta}x_{b'}x_{b}}\right)=-4(Q+2)|x|^{2}.
\end{aligned}
\end{equation}
We can write $g=\phi^{\frac{4}{Q-2}}g_{0}$ with $\phi={\|\xi\|^{-\frac{Q-2}{2}}}.$ It follows from the transformation formula (\ref{122}) of the scalar curvatures that
\begin{equation*}\begin{aligned}%\label{8899}
s_{g}&=\phi^{-\frac{Q+2}{Q-2}}b\Delta_{0}\phi
=\|\xi\|^{\frac{Q+2}{2}}b\Delta_{0}
\|\xi\|^{-\frac{Q-2}{2}}\\
&=-\frac{Q-2}{8}b\|\xi\|^{\frac{Q+2}{2}}\left(\frac{Q+6}{8}
\|\xi\|^{-\frac{Q+14}{2}}\left|\nabla_{0}\|\xi\|^{4}\right|^{2}+\|\xi\|^{-\frac{Q+6}{2}}
\Delta_{0}\|\xi\|^{4}\right)=(Q-2)(Q-1)\frac{|x|^{2}}{\|\xi\|^{2}}
\end{aligned}\end{equation*}
by (\ref{A2}) and (\ref{A1}) for $b=4\frac{Q-1}{Q-2}$.
\end{proof}
\begin{prop}\label{psum}
If $t$ is sufficiently small, we have $\lambda(M_{t,A},g,\mathbb{I}_{t,A})>0.$
\end{prop}

\begin{proof}
See  \cite{Kobayashi} for the Riemannian case, \cite{wang1}  for the spherical CR case (see also \cite{Cheng} for a different  proof for the spherical CR case) and \cite{Shi} for the spherical  qc case. In \cite{CCH,Dietrich}, it is generalized to the  non-spherical CR case. The proof of this proposition is similar to the spherical CR case and the spherical  qc case.  Let
$M_{0}=M_{1}\backslash\{\xi_{1}\}\cup M_{2}\backslash\{\xi_{2}\},$
and let $\hat{g}$ be an OC metric on $M_{0}$. Then, by multiplying a positive function $\mu\in C^{\infty}(M)\backslash\{\xi_{1},\xi_{2}\}$, we
can assume $g=\mu\hat{g}$ satisfying
\begin{align*}
\left(\psi_{i}^{-1}\right)^{*}g|_{\xi}=\frac{g_{0}|_{\xi}}{\|\xi\|^{2}}\quad {\rm on} \quad B(0,2)\backslash\{0\},
\end{align*}
where $\psi_{i}:U_{i}\rightarrow B(0,2),\ i=1,2,$ are coordinate charts in  $(\ref{265}).$ It is easy to see that gluing mapping
$\Psi_{t,A}$ in (\ref{2.60}) preserves the metric $\frac{g_{0}|_{\xi}}{\|\xi\|^{2}}$ on $t<\|\xi\|<\frac{1}{t}$, $0<\frac{2}{3}t<1$, by the transformation formula and Proposition \ref{o2.9}. $\frac{g_{0}|_{\xi}}{\|\xi\|^{2}}$ is invariant under the rotation $A$   and the inversion $R.$ Hence we can glue $g$ by $\Psi_{t,A}$ to obtain a  OC metric that coincides with $g$  on $M_{1}\backslash \overline{U_{1}(t)}\cup
M_{2}\backslash \overline{U_{2}(t)}.$  We denote the resulting OC metric also by $g$ by abuse of notations, and  the connected sum by
$(M_{t},g,\mathbb{I}_{t}).$ Here we omit the subscripts $A$ for simplicity.

$(M_{0},g,\mathbb{I})$ has two cylindrical ends. We can identify the ball with cylindrical end by the mapping
\begin{equation}
\begin{aligned}
\Psi:\quad B(0,1)\longrightarrow[0,\infty)\times\Sigma,\quad
\xi=D_{e^{-u}}(\eta)\longmapsto\left(\ln\frac{1}{\|\xi\|},\frac{\xi}{\|\xi\|}\right)
=(u,\eta),
\end{aligned}
\end{equation}
where $\Sigma=\{\eta\in\mathscr{H};\|\eta\|=1\}$ is diffeomorphic to the sphere $S^{14}.$
Define a Carnot-Carath\'eodory metric
$$\left.\tilde{g}\right|_{\Psi(\xi)}=(\Psi^{-1})^{*}\left(\frac{g_{0}
|_{\xi}}{\|\xi\|^{2}}\right)$$
on $[0,\infty)\times\Sigma$ and  $\tilde{\Theta}=\left(\Psi^{-1}\right)^*\Theta_0= (\tilde{\theta}_{1},\cdots )$ is a compatible contact form.
$\left(B(0,1)\backslash\{0\},\frac{g_{0}}{\|\xi\|^{2}},\mathbb{I}\right)$ is OC equivalent to $([0,\infty)\times\Sigma,\tilde{g},\mathbb{I})$.
Since $(\psi_{i}^{-1})^{*}\hbox{d}V_{g}=\frac{\hbox{d}V_{0}}{\|\xi
\|^{Q}}$ is invariant under rescaling, it is easy to see that the measure $\tilde{\theta}_{1}\wedge\cdots\wedge\tilde{\theta}_{7}
\wedge(\hbox{d}\tilde{\theta}_\beta)^{4}$ is invariant under translation $(u',\xi)\rightarrow(u'+u_{0},\xi)$ on $[0,\infty)\times\Sigma$. As a measure, we have
\begin{align}\label{4.12}
\tilde{\theta}_{1}\wedge\cdots
\wedge\tilde{\theta}_{7}\wedge(\hbox{d}\tilde{\theta}_\beta)^{4}
=\hbox{d}u\hbox{d}S_{\Sigma},
\end{align}
where $\hbox{d}S_{\Sigma}$ is a measure on $\Sigma$. Set $l=\ln\frac{1}{t},$
and write
\begin{align*}
(M_{0},g,\mathbb{I})=([0,\infty)\times\Sigma,\tilde{g},\mathbb{I})
\cup(\hat{M},g,\mathbb{I})
\cup([0,\infty)\times\Sigma,\tilde{g},\mathbb{I}),
\end{align*}
where $\hat{M}=M\backslash (U_{1}(1)\cup U_{2}(1))$. We identify two pieces of  $((0,l)\times\Sigma,\tilde{g},\mathbb{I})$  to get
$(M_{t},g,\mathbb{I}_{t})$.

For $\eta\in\Sigma,$  we can write $\eta=(x_{\eta},t_{\eta})\in\mathscr{H}$ for some $x_\eta\in\mathbb{R}^{8},t_{\eta}\in\mathbb{R}^{7}.$ Then,
\begin{align}\label{4.4}
|\nabla_{\tilde{g}}u|^2=\left|\nabla_{g}\left(\ln\frac{1}{\|\xi\|}\right)\right|^2
=\|\xi\|^{2}\left|\nabla_{0}\left(\ln\frac{1}{\|\xi\|}\right)\right|^{2}
=\frac{1}{16\|\xi\|^{6}}\left|
\nabla_{0}\|\xi\|^{4}\right|^{2}=\frac{|x_\xi|^{2}}{\|\xi\|^{2}}=|x_\eta|^2,
\end{align}
by (\ref{A2}), where $\xi=(x_\xi,t_\xi)=\Psi^{-1}(u,\eta)\in B(0,1)$ for some $x_\xi,t_\xi,$ and
\begin{align}\label{4.5}
s_{\tilde{g}}(u,\eta)=(Q-2)(Q-1){|x_{\eta}|^{2}},
\end{align}
by   Lemma \ref{ls}.
By the definition of the OC Yamabe invariant $\lambda(M_{t},g,\mathbb{I}_{t})$, we can find a positive function $f_{l}\in C^{\infty}(M_{t})$ such that
\begin{align}\label{4.14}
\int_{M_{t}}\left(b|\nabla_{g}f_{l}|^{2}+s_{g}f_{l}^{2}\right) \hbox{d}V_{g}
<\lambda(M_{t},g,\mathbb{I}_{t})+\frac{1}{l},\quad {\rm with}\quad
\int_{M_{t}}f_{l}^{\frac{2Q}{Q-2}}\hbox{d}V_{g}=1.
\end{align}
Put $A_{1}=-\min\left\{0,\min_{x\in \hat{M}}s_{\hat{g}}\right\}{\rm Vol}(\hat{M})^{\frac{4}{Q-2}},$ which is uniformly bounded by ${\rm
Vol}(M,g)$. Thus  by using H\"older's inequality we get from (\ref{4.14}) that
$$\int_{[0,l]\times\Sigma}\left(b|
\nabla_{\tilde{g}}f_{l}|^{2}+s_{\tilde{g}}f_{l}^{2}\right)
\hbox{d}u\hbox{d}S_{\Sigma}<\lambda(M_{t},g,\mathbb{I}_{t})+
\frac{1}{l}+A_{1},$$
(cf. Lemma 6.2 in \cite{Kobayashi}). Note that $s_{\tilde{g}}$ is nonnegative on $[0,\infty)\times\Sigma$ by (\ref{4.5}). Therefore,
there exists $l_{*}\in[1,l-1]$ such that
$$\frac12\int_{l_{*}-1}^{l_{*}+1}\hbox{d}u\int_{\Sigma}\left(b|
\nabla_{\tilde{g}}f_{l}|^{2}+s_{\tilde{g}}f_{l}^{2}\right)
\hbox{d}{S_{\Sigma}}<\frac{\lambda(M_{t},g,\mathbb{I}_{t})+
\frac{1}{l}+A_{1}}{l-2},$$
i.e. we have the estimate
\begin{align}\label{4.16}
\int_{l_{*}-1}^{l_{*}+1}\hbox{d}u\int_{\Sigma}\left(|
\nabla_{\tilde{g}}f_{l}(u,\eta)|^{2}+|x_{\eta}|^{2}
f_{l}^{2}(u,\eta)\right)\hbox{d}S_{\Sigma}(\eta)<\frac{C}{l},
\end{align}
by the scalar curvature of $\tilde g$ in (\ref{4.5}), where $C$ is a constant independent of $l$ (because the OC Yamabe invariants
$\lambda(M_{t},g,\mathbb{I}_{t})$ for $t>1$ have a uniform upper bound by choosing a test function). It is different from the Riemannian case that the
scalar curvature of  $\frac{g_{0}|_{\xi}}{\|\xi\|^{2}}$ is not constant. But it is still independent of the variable $u.$ Now define a Lipschitz function
$F_{l}$ on $M_{0}$ by $F_{l}=f_{l}$ on $[0,l_{*})\times\Sigma\cup\hat{M}\cup[0,l-l_{*})\times\Sigma$ and
\begin{equation}
F_{l}(u,x)=\left\{
\begin{array}{rcl}
&(l_{*}+1-u){f_{l}}(u,x)& \quad {\rm for}\quad(u,x)\in[l_{*},l_{*}+1]\times\Sigma,\\
&0 &\ \ \ {\rm for}\quad(u,x)\in[l_{*}+1,\infty)\times\Sigma,
\end{array}
\right.
\end{equation}
 and similarly on $[l-l_{*} ,\infty)\times\Sigma$.

By definition, $|\nabla_{\tilde{g}}F_{l}|=|\nabla_{\tilde{g}}f_{l}|$ and  $F_{l}^{2}=f_{l}^{2}$ hold on
$[0,l_{*})\times\Sigma\cup\hat{M}\cup[0, l-l_{*} )\times\Sigma.$ On the other hand, note that
$|
\nabla_{\tilde{g}}F_{l}|\leq |
\nabla_{\tilde{g}}u||{f}_{l}|+|
\nabla_{\tilde{g}}{f}_{l}|$ pointwisely on $(l_{*},l_{*}+1)\times\Sigma$ by definition.
By (\ref{4.4}),  (\ref{4.5}) and estimate (\ref{4.16}),
we find that
\begin{equation*}
\begin{aligned}
&\int_{(l_{*},l_{*}+1)\times\Sigma}\left(b|
\nabla_{\tilde{g}}F_{l}|^{2}
+s_{\tilde{g}}F_{l}^{2}\right)
\hbox{d}u\hbox{d}S_{\Sigma}\leq C'\int_{(l_{*},l_{*}+1)\times\Sigma}
\left(|\nabla_{\tilde{g}}{f}_{l}|^{2}+|x_{\eta}|^{2}
{f_{l}}^{2}\right)\hbox{d} u\hbox{d} S_{\Sigma}(\eta)\leq\frac{B'}{l}.
\end{aligned}
\end{equation*}
Therefore, we get
\begin{align*}
\int_{M_{0}}\left(b|\nabla_{g} F_{l}|^{2}+s_{g}F_{l}^{2}\right)
\hbox{d}V_{g}<\lambda(M_{t},
g,\mathbb{I}_{t})+\frac{B}{l},
\end{align*}
for some constant  $B$ independent of $l$.

Obviously from (\ref{4.14}) and the definition of $F_{l},$ we get $\int_{M_{0}}F_{l}^{\frac{2Q}{Q-2}}\hbox{d}V_{g}>1.$
 Therefore,
\begin{align}\label{4.20}
\inf_{F>0}\frac{\int_{M_{0}}\left(b|\nabla_{g}F|^{2}
+s_{g}\right)\hbox{d}V_{g}}{\left(\int_{M_{0}}
F^{\frac{2Q}{Q-2}}\hbox{d}V_{g}\right)^{\frac{Q-2}{Q}}}<
\lambda(M_{t},g,\mathbb{I}_{t})+\frac{B}{l},
\end{align}
where the infimum is taken over all nonnegative Lipshitzian functions with compact
support. It follows from the definition of the Yamabe invariant that the left side is greater than or equal to
$\lambda(M,g,\mathbb{I})$.
Then $\lambda(M_{t},g,\mathbb{I}_{t})$ is positive if $l$ is sufficiently
large, i.e. $t$ is sufficiently small. We  complete the proof.
\end{proof}

\section{The convex cocompact discrete subgroups of ${\rm F}_{4(-20)}$}
\subsection{Convex cocompact subgroups of ${\rm F}_{4(-20)}$}
A group $G$ is called \emph{discrete} if the topology on $G$ is  discrete. We say that $G$ acts \emph{discontinuously} on a space $X$ at point $w$ if there is a neighborhood $U$ of $w$, such that $g(U)\cap U=\emptyset$ for all but finitely many $g\in G$.

Let $\Gamma$ be a discrete subgroup of ${\rm F}_{4(-20)}$.  It is known that the limit set $\Lambda(\Gamma)$ in (\ref{limitset}) does not depend on the choice of $q\in \mathcal U$ (cf. \cite[Proposition 1.4, 2.9]{Eberlein}).
The limit set $\Lambda(\Gamma)$ of all limit points is closed and invariant under $\Gamma$.
The \emph{radial limit set of $\Gamma$}  is $$\Lambda^{r}(\Gamma):=\left\{\xi\in\Lambda(\Gamma)\left|
\liminf_{T\rightarrow\infty}d(\xi_{T},\gamma(0))<\infty,\gamma\in \Gamma\right.\right\},$$
where $\xi_{T}$ refers to the point on the ray from $0$ to $\xi$ for which $d(0,\xi_{T})=T$ and $d(\cdot,\cdot)$ is the octonionic hyperbolic distance.
$\Gamma$ is called a \emph{Kleinian group} if $\Omega(\Gamma)$ defined in (\ref{Omegag}) is non-empty. A Kleinian group is called \emph{elementary} if $\Lambda(\Gamma)$ contains at most two points.
$\Gamma$ is called \emph{convex cocompact} if $\bar{M}_{\Gamma}=(\mathcal U\cup\Omega(\Gamma))/\Gamma$ is a compact manifold with boundary. In this case, $\Omega(\Gamma)/\Gamma$ is a compact smooth OC manifold.

\begin{prop}\label{cor}{\rm{(cf.  \cite[p. 528]{Corlette})}}
Suppose that $\Gamma $ is a convex cocompact group of ${\rm F}_{4(-20)}$. Then,\\
{\rm(1)} The radial limit set coincides with the limit set.\\
{\rm(2)} Any small deformation of the inclusion $\iota:\Gamma\rightarrow {\rm F}_{4(-20)}$ maps $\Gamma$ isomorphically to a convex cocompact group.
\end{prop}
\subsection{The Patterson-Sullivan measure}
\begin{thm}{\rm (cf.  \cite[p. 532]{Corlette})}
For any convex cocompact  subgroup $\Gamma$ of ${\rm F}_{4(-20)}$, there exists a probability measure $\tilde\mu_{\Gamma}$ supported on 	 $\Lambda(\Gamma)$ such that
\begin{align}\label{5.5}
\gamma^{*}\tilde\mu_{\Gamma}=|\gamma'|^{\delta(\Gamma)}\tilde\mu_{\Gamma}
\end{align}
for any $\gamma\in\Gamma,$ where $|\gamma'|$ is a conformal factor.
\end{thm}
See \cite{EMM, wang1} for Patterson-Sullivan measure for the complex case and \cite{Shi} in the quaternionic case. We need to know  the explicit conformal  factor $|\gamma'|$ for our purpose later.
Fix a reference point $v=(0,-1)\in \mathcal U,$ where $\tilde v=(-1,0,1)^t.$  Let us recall the definition of Patterson-Sullivan measure in  \cite{Patterson}.
Define a family of measures as
\begin{align*}
\tilde\mu_{s,z}:=\frac{\sum_{\gamma\in \Gamma}e^{-\frac{1}{2}s\cdot d(z,\gamma (w))}\delta_{\gamma (w)}}{\sum_{\gamma\in \Gamma}e^{-\frac{1}{2}s\cdot d(v,\gamma (w))}},\quad
\end{align*} for some fixed $w\in\mathcal U,$
where $d(\cdot,\cdot)$ is the hyperbolic distance defined in (\ref{dis sie}), $\delta_{\gamma (w)}$ is the Dirac measure supported at point $\gamma(w)$ and $v=(0,-1)$. For each  $s>\delta(\Gamma)$, this  is  a  finite positive measure  concentrated  on $\Gamma w\subset\overline{\Gamma w}$. Since the  set  of all  probability  measures  on $\overline{\Gamma w}$ is  compact (cf.  \cite[p. 532]{Corlette}),    there  is  a  sequence $s_{i}$ approaching $\delta(\Gamma)$ from  above  such  that $\tilde\mu_{s_{i},z}$ approaches  a  limit $\tilde\mu_{s,z}$. After  rewriting  the coefficients,  we  may  assume  that  the  denominator  in  the  definition  of $\mu_{s,z}$ diverges at $s=\delta(\Gamma)$. Thus,  we  replace  the  above  expression  by
\begin{align*}
\tilde\mu_{s,z}=\frac{\sum_{\gamma\in \Gamma}a_{\gamma}e^{-\frac{s}{2}\cdot d(z,\gamma (w))}\delta_{\gamma (w)}}{L(s,v)},\quad {\rm where}\ L(s,v)=\sum_{\gamma\in \Gamma}a_{\gamma}e^{-\frac{s}{2}\cdot d(v,\gamma (w))},
\end{align*}
with  the $a_{\gamma}$ chosen  so  that  the  denominator  converges  for
$s>\delta(\Gamma)$  and  diverges  for $s\leq\delta(\Gamma)$.  The definition of the measure $\tilde\mu_{s}$ does not depend on $w\in\mathcal U$ and the choice of $a_{\gamma}$ (cf. \cite[p. 532]{Corlette}). The \emph{Patterson-Sullivan measure} is the weak limit of these measures:
\begin{align*}
\tilde\mu_{\Gamma,z}=\lim_{s_{i}\rightarrow\delta(\Gamma)^{+}}\tilde\mu_{s_{i},z}.
\end{align*}
For any $\gamma \in {\rm F}_{4(-20)}$ and any $f\in C(\bar{\mathcal U})$, we have
\begin{equation*}
\begin{aligned}
(\gamma^{*}\tilde\mu_{s_{i},z})(f)&=\frac{\sum_{\widetilde \gamma\in\Gamma}a_{\widetilde \gamma}
e^{-\frac{s_{i}}{2}\cdot d(z,\widetilde \gamma (w))}\gamma^{*}\delta_{\widetilde \gamma (w)}(f)}{L(s_{i},v)}\\&=\frac{\sum_{\widetilde\gamma\in \Gamma}a_{\widetilde \gamma}e^{-\frac{s_{i}}{2}\cdot d(\gamma^{-1} (z),\gamma^{-1}(\widetilde \gamma (w)))}f(\gamma^{-1}{(\widetilde \gamma (w))})}{L(s_{i},v)} \\&=\frac{\sum_{\widetilde \gamma\in \Gamma}a_{\gamma\widetilde \gamma}e^{-\frac{s_{i}}{2}\cdot d(\gamma^{-1} (z),\widetilde \gamma (w))}f({\widetilde \gamma (w)})}{L(s_{i},v)}=\tilde\mu_{s_{i},\gamma^{-1} (z)}(f)
\end{aligned}
\end{equation*}
by the invariance of the octonionic hyperbolic distance $d(\cdot,\cdot)$  under the action of ${\rm F}_{4(-20)}.$ It is easy to see that $\{a_{ \gamma\widetilde\gamma}\}$ is also such sequence satisfying  the definition for fixed $\gamma$.
Let $s_{i}\rightarrow \delta^{+}$. We get
\begin{align*}
\gamma^{*}\tilde\mu_{\Gamma,z}=\tilde\mu_{\Gamma,\gamma^{-1} (z)}.
\end{align*}
The {\it Buseman function} is defined by \begin{align}
b_{\xi}(x)=\lim_{t\rightarrow\infty}(d(x,\sigma(t))-t),
\end{align}
where $\sigma:[0,\infty]\longrightarrow H_{\mathbb{O}}^{2}$ is a geodesic ray asymptotic to $\xi.$
Note that $$\left(\gamma(z)_{1},\gamma(z)_{2},
\gamma(z)_{3}\right)^t\sim\left(
\gamma(\tilde{z})_{1}\gamma(\tilde{z})_{3}^{-1},
\gamma(\tilde{z})_{2}\gamma(\tilde{z})_{3}^{-1},1\right)^t.$$
Recall that we have the following the Radon-Nikodym relation (cf. \cite[p. 77, 81]{Yue}):
\begin{align*}
\left.\frac{d\tilde\mu_{\Gamma,\gamma^{-1} (z)}}{d\tilde\mu_{\Gamma,z}}\right|_{\xi}=
e^{\frac{\delta}{2}(b_{\xi}(z)-b_{\xi}(\gamma^{-1}(z)))}
\end{align*}
and
\begin{equation*}\begin{aligned}
b_{\xi}(z)-b_{\xi}(w)=&\lim_{t\rightarrow\infty}\left(d(z,\sigma(t))-d(w,\sigma(t))\right) =2\lim_{t\rightarrow\infty}\left(\cosh^{-1}\left(z,{\sigma(t)}\right)- \cosh^{-1}\left(w,{\sigma(t)}\right)\right)\\ =&2\lim_{t\rightarrow\infty}\ln\frac{(z,\sigma(t))+\sqrt{(z,\sigma(t))^2 -1}}{(w,\sigma(t))+\sqrt{(w,\sigma(t))^2 -1}}=2\lim_{t\rightarrow\infty} \ln\frac{(z,\sigma(t))}{(w,\sigma(t))}\\=&2\ln\frac{|\tilde{z}^*D_{1}\tilde \xi|
|\tilde{w}^*D_{1}\tilde{w}|^{\frac{1}{2}}}
{|\tilde{w}^*D_{1}{\tilde \xi}||\tilde{z}^*D_{1}\tilde{z}|^{\frac{1}{2}}},
\end{aligned}\end{equation*}
for $\xi\in \partial\mathcal U.$  The second last identity holds by $\cosh^{-1}s=\ln\left(s+\sqrt{s^2-1}\right)$ for $s>0$ and $(z,\sigma(t)),(w,\sigma(t))\rightarrow\infty.$

 For $z=( 0,-1),\eta=(\eta_1,\eta_2)\in \mathcal U,$ denote
\begin{align}\label{p1}
\varphi(\eta):=\left|{\tilde z}^*D_1\tilde\eta\right|^2,\quad \chi(\eta):=\varphi(\eta)^{\frac{\delta(\Gamma)}{2}},
\end{align}where $\tilde z=(-1, 0,1)^t,\tilde\eta=(\eta_2,\eta_1,1)^t$ and $D_1$ is given by (\ref{d1}).
Then we have \begin{align} \varphi(\gamma(\eta))=\left|{\tilde z}^*D_1\widetilde{\gamma(\eta)}\right|^2.
\end{align}
Then  we have
\begin{equation}\label{distance}
\begin{aligned}
\left.\frac{{\rm d}\tilde\mu_{\Gamma,{\gamma^{-1}(z)}}}{{\rm d}\tilde\mu_{\Gamma,{z}}}\right|_{\xi}
&=\lim_{\eta\rightarrow\xi}\left(\frac{|\tilde{z}^*D_{1}\tilde \eta|
\left|\widetilde{\gamma^{-1}(z)}^*D_{1}\widetilde{\gamma^{-1}(z)}\right|^{\frac{1}{2}}}
{\left|\widetilde{\gamma^{-1}(z)}^*D_{1}{\tilde{\eta}}\right||\tilde{z}^* D_{1}\tilde{z}|^{\frac{1}{2}}
}\right)^{\delta(\Gamma)}\\&=\lim_{\eta\rightarrow\xi}
\left|\frac{(z,\eta)}{(\gamma^{-1} (z),\eta)}\right|^{\delta(\Gamma)}
=\lim_{\eta\rightarrow\xi}\left|\frac{(z,\eta)}{( z,\gamma(\eta))}\right|^{\delta(\Gamma)}\\
&=\lim_{\eta\rightarrow\xi}\left|\frac{\frac{\left|{\tilde z}^*D_1\tilde\eta\right|^2}{\left|{\tilde \eta}^*D_1\tilde\eta\right|}}{\frac{\left|{\tilde z}^*D_1\widetilde{\gamma(\eta)}\right|^2}{\left|\widetilde{\gamma(\eta)}^*D_1 \widetilde{\gamma(\eta)}\right|}}\right|^{\frac{\delta(\Gamma)}{2}}=\frac{1}{
|\gamma(\tilde{\xi})_3|^{\delta(\Gamma)}}\left|\frac{\varphi(\xi)}{\varphi(\gamma(\xi))} \right|^{\frac{\delta(\Gamma)}{2}},
\end{aligned}
\end{equation}
where $\eta\in \mathcal U$ and  $|(\cdot,\cdot)|=\cosh\left(\frac{1}{2}d(\cdot,\cdot)\right)$  is invariant under ${\rm F}_{4(-20)}.$
The last identity  holds since \begin{align}\label{p2}
\frac{\left|\widetilde{\gamma(\eta)}^*D_1 \widetilde{\gamma(\eta)}\right|}{\left|{\tilde \eta}^*D_1\tilde\eta\right|}=\frac{2{\rm Re}\left[{{\gamma(\tilde{\eta})_1}\overline{\gamma(\tilde{\eta})_3}}\right]+
\left|\gamma(\tilde{\eta})_2\right|^2}{\left|\gamma(\tilde\eta)_3\right|^2\left(2{\rm Re}\ \eta_2+|\eta_1|^2\right)}=\frac{1}{\left|\gamma(\tilde\eta)_3\right|^2},
\end{align}
by $$2{\rm Re}\left[{{\gamma(\tilde{\eta})_1}\overline{\gamma(\tilde{\eta})_3}}\right]+
\left|\gamma(\tilde{\eta})_2\right|^2= 2{\rm Re}\ \eta_2+|\eta_1|^2,$$ which can be checked directly for all dilation $D_\delta,$ left translations $\tau_{(y,s)},$   rotation $S_\mu$  and  inversion $R$ in (\ref{a1})-(\ref{a2}).
So if we  define $\mu_{\Gamma}:=\mu_{\Gamma,z}=\chi\tilde\mu_{\Gamma,z} $ with $\chi$ given by (\ref{p1}), we have
\begin{align}\label{gg}
\gamma^{*}\hbox{d}\mu_{\Gamma}(\xi)=\frac{1}{
|\gamma(\tilde{\xi})_3|^{\delta(\Gamma)}} \hbox{d}\mu_{\Gamma}(\xi).
\end{align}

Note that
$\frac{1}{|\gamma(\tilde{\xi})_3|}$ coincides   with $\phi$ given in (\ref{o2.9}), i.e. $\gamma^*g_0=\phi^2g_0$ with $\phi=\frac{1}{|\gamma(\tilde{\xi})_3|}.$ For example, let $\gamma$ be the inversion given in (\ref{a1}), then we  have $$
\gamma(\tilde \xi)=\begin{pmatrix}0&0&1\\0&-1&0\\1&0&0\end{pmatrix} \begin{pmatrix}\xi_2\\\xi_1\\1\end{pmatrix}=\begin{pmatrix}1\\-\xi_1\\\xi_2 \end{pmatrix},$$ i.e. $\left|\gamma(\tilde{\xi})_3\right|^2=\left|\xi_2\right|^2 =\left||x|^2-t\right|^2=|x|^4+|t|^2$ by (\ref{sth}).
\begin{rem}
In the CR and qc cases \cite{wang1,Shi}, we use the ball model and choose $z$ to be the origin. Then $\mu_{\Gamma,z}$ automatically has conformal factor as the action of the group on the sphere.  But in the OC case, we use the flat model on which we have an extra conformal factor $\chi$ given by (\ref{p1}). But the flat model has many advantages.
\end{rem}

\subsection{An invariant OC metric of Nayatani type}

When the OC manifold is $\Omega(\Gamma)/\Gamma$ for some convex cocompact subgroup $\Gamma$ of ${\rm F}_{4(-20)}$, we can construct an invariant
OC metric $g_{\Gamma}$, which is the OC generalization of Nayatani's canonical metric in conformal geometry \cite{Nayatani}. See \cite{wang1} and \cite{Shi} for CR case and qc case, respectively.

Recall $g_{\Gamma}=\phi_{\Gamma}^{\frac{4}{Q-2}} g_{0}$ defined in (\ref{ggamma}).
Since
\begin{align*}
G_{0}(\gamma(\xi),\gamma(\zeta))=|\gamma(\tilde{\xi})_3|
^{\frac{Q-2}{2}}|\gamma(\tilde{\zeta})_3|^{\frac{Q-2}{2}}
G_{0}(\xi,\zeta)
\end{align*}
by the conformal factor  (\ref{gg}),  Proposition \ref{o2.9} and the transformation formula (\ref{32}) of the  Green functions, we have
\begin{equation}\label{6.3}
\begin{aligned}
\phi_{\Gamma}(\gamma(\xi))&= \left(\int_{\Lambda({\Gamma})}G_{0}
^{\frac{2\delta(\Gamma)}{Q-2}}(\gamma(\xi),\zeta)
\hbox{d}{\mu_{\Gamma}(\zeta)}\right)^{\frac{Q-2}{2\delta(\Gamma)}}
=\left(\int_{\Lambda({\Gamma})}G_{0}^
{\frac{2\delta(\Gamma)}{Q-2}}(\gamma(\xi),\gamma(\zeta))
\hbox{d}\gamma^{*}{\mu_{\Gamma}(\zeta)}\right)^{\frac{Q-2}{2\delta(\Gamma)}}\\
&=\left(\int_{\Lambda({\Gamma})} \left|\gamma(\tilde\xi)_3\right|^{\delta({\Gamma})} G_{0}^{\frac{2\delta(\Gamma)}{Q-2}} (\xi,\zeta) \hbox{d}{\mu_{\Gamma}(\zeta)}\right)^{\frac{Q-2}{2\delta(\Gamma)}}
=\left|\gamma(\tilde\xi)_3\right|^{\frac{Q-2}{2}} \phi_{\Gamma}(\xi).
\end{aligned}
\end{equation}
Therefore, (\ref{6.3}) together with Proposition \ref{o2.9} and  (\ref{gg}) implies that
$\gamma^{*}g_{\Gamma}=g_{\Gamma}.$
So it induces an OC metric  on the compact OC manifold $\Omega(\Gamma)/\Gamma$.

The proof of Theorem \ref{6.1}  is similar to the CR case \cite[p. 265]{wang1} and qc case  \cite[p. 302]{Shi}, we omit  details.

\end{document}